\documentclass{amsart}
\usepackage[francais,english]{babel}
\usepackage[T1]{fontenc}
\usepackage{amssymb,amsmath,dsfont}
\usepackage{amsrefs}

\newtheorem{theorem}{Theorem}
\newtheorem{proposition}{Proposition}
\newtheorem{corollary}{Corollary}

\newtheorem{lemma}{Lemma}

\newcommand{\B}{\mathbb{B}}

\newcommand{\R}{\mathbb{R}}

\def \crit {2^\star(s)}

\title{Optimal Hardy-Sobolev inequalities on compact Riemannian manifolds}
\author{Hassan Jaber}
\address{Universit\'e de Lorraine, Institut Elie Cartan de Lorraine, UMR 7502, Vand\oe uvre-l\`es-Nancy, F-54506, France.}
\email{Hassan.Jaber@univ-lorraine.fr}
\date{January 12th, 2014.}

\begin{document}
\maketitle
\begin{abstract}
Given a compact Riemannian Manifold $(M,g)$ of dimension $n\geq 3$, a point $x_0\in M$ and $s\in (0,2)$, the Hardy-Sobolev embedding yields the existence of $A,B>0$ such that
\begin{equation}\label{ineq:abstract}
\left(\int_M \frac{|u|^{\frac{2(n-s)}{n-2}}}{d_g(x,x_0)^s}\, dv_g\right)^{\frac{n-2}{n-s}}\leq A\int_M |\nabla u|_g^2\, dv_g+B\int_M u^2\, dv_g
\end{equation}
for all $u\in H_1^2(M)$. It has been proved in Jaber \cite{jaber:test:fcts} that $A\leq K(n,s)$ and that one can take any value $A>K(n,s)$ in \eqref{ineq:abstract}, where $K(n,s)$ is the best possible constant in the Euclidean Hardy-Sobolev inequality. In the present manuscript, we prove that one can take $A=K(n,s)$ in \eqref{ineq:abstract}.\end{abstract}

Let $(M,g)$ be a smooth compact Riemannian Manifold of dimension $n \geq 3$ without boundary, $d_g$ be the Riemannian distance on $M$ and $H_1^2(M)$ be the Sobolev space defined as the completion of $C^\infty(M)$ for the norm $u \mapsto \Vert u\Vert_2+\Vert\nabla u\Vert_2$. We fix $x_0 \in M, s\in (0,2)$ and let $\crit:=\frac{2(n-s)}{n-2}$ be the critical Hardy-Sobolev exponent. We endow the weighted Lebesgue space $L^p(M, d_g(\cdot, x_0)^{-s})$ with its natural norm $u\mapsto \Vert u\Vert_{p,s}:=\left(\int_M |u|^p d_g(\cdot, x_0)^{-s}\, dv_g\right)^{\frac{1}{p}}$. It follows from the Hardy-Sobolev inequality that the Sobolev space $H_1^2(M)$ is continuously embedded in the weighted Lebesgue space $L^p(M, d_g(\cdot, x_0)^{-s})$ if and only if $1\leq p\leq \crit$, 
and that this embedding is compact if and only if $1\leq p<\crit$.  From the embedding of $H_1^2(M)$ in $L^{\crit}(M, d_g(\cdot, x_0)^{-s})$, one gets the existence of $A,B>0$ such that  
\begin{equation}\label{eqt-0.1-Main-Th-Sharp-Cst-H-S-ineq}
\|u\|_{\crit,s}^2 \leq A\|\nabla u\|_2^2 + B\|u\|_2^2
\end{equation}
for all $u \in H_1^2(M)$. We let  $K(n,s)$ be the optimal constant of the Euclidean Hardy-Sobolev inequality, that is
\begin{equation}\label{eqt-0.2-Main-Th-Sharp-Cst-H-S-ineq}
K(n,s)^{-1} := \inf_{\varphi \in C_c^\infty(\R^n)\setminus\{0\}}
\frac{\int_{\R^n}|\nabla\varphi|^2 dX}{\left(\int_{\R^n}\frac{|\varphi|^{2^\star(s)}}{|X|^s}dX\right)^{\frac{2}{\crit}}}
\end{equation}
we denote by $C_c^\infty(\R^n)$ the set of $C^\infty$-smooth functions on $\R^n$ with compact support and by $|\cdot|$ the Euclidean norm in  $\R^n$. The value of $K(n,s)$ is 
$$K(n,s) = \left[(n-2)(n-s)\right]^{-1}\left(\frac{1}{2-s}\omega_{n-1}\frac{\Gamma^2(n-s/2-s)}{\Gamma(2(n-s)/2-s)}\right)^{-\frac{2-s}{n-s}},$$
where $\omega_{n-1}$ is the volume of the unit sphere on $\R^n$ and $\Gamma$ is the Euler function. It was computed independently by Aubin \cite{Aubin-1}, Rodemich \cite{Rodemich}  and 
Talenti \cite{Talenti} for the case $s=0$, and the value for $s\in (0,2)$ has been computed by Lieb (Theorem 4.3 in \cite{Lieb-1}). 
Following Hebey \cite{Hebey-3}, we define $A_0(M,g,s)$ to be the best first constant of the Riemannian Hardy-Sobolev inequality, that is
\begin{equation}\label{eqt-0.3-Main-Th-Sharp-Cst-H-S-ineq}
A_0(M,g,x_0,s) := \inf\{A > 0\, \exists B>0\hbox{ such that }\eqref{eqt-0.1-Main-Th-Sharp-Cst-H-S-ineq} \hbox{ holds for all }u\in H_1^2(M)\}.
\end{equation}
For the Sobolev inequality (\eqref{eqt-0.1-Main-Th-Sharp-Cst-H-S-ineq} when $s=0$), Aubin proved in \cite{Aubin-1} that $A_0(M,g,x_0,0) = K(n,0)$. When $s\in (0,2)$, the author proved in \cite{jaber:test:fcts} that $A_0(M,g,x_0,s) = K(n,s)$. 
In particular, for any $\epsilon > 0$, there exists $B_\epsilon > 0$ such that we have : 
\begin{equation}\label{eqt-0.4-Main-Th-Sharp-Cst-H-S-ineq}
\|u\|_{\crit,s}^2 \leq (K(n,s) + \epsilon)\int_M|\nabla u|_g^2dv_g + B_\epsilon\int_Mu^2dv_g
\end{equation}
for all $u\in H_1^2(M)$. See  Thiam \cite{Elhadji-A.T.} for a version with an additional remainder. 
The constant $B_\epsilon$ obtained in \cite{jaber:test:fcts} goes to $+\infty$ as $\epsilon\to 0$, and therefore the method used in \cite{jaber:test:fcts} does not allow to take $\epsilon=0$ in \eqref{eqt-0.4-Main-Th-Sharp-Cst-H-S-ineq}.

\medskip\noindent A natural question is to know whether the infimum $A_0(M,g,x_0,s)$ is achieved or not, that is if there exists $B > 0$ such that inequality \eqref{eqt-0.1-Main-Th-Sharp-Cst-H-S-ineq} holds for all $u\in H_1^2(M)$ with $A = K(n,s)$.
We prove the following :

\begin{theorem}\label{Main-Th-Sharp-Cst-H-S-ineq}
Let $(M,g)$ be a smooth compact Riemannian Manifold of dimension $n \geq 3$, $x_0 \in M$, and $s \in (0,2)$. We let $\crit := \frac{2(n-s)}{n-2}$ be the critical Hardy-Sobolev exponent. 
Then there exists $ B_0(M,g,s,x_0)>0$ depending on $(M,g)$ and $s$ such that
\begin{equation}\label{eqt-1.0-Main-Th-Sharp-Cst-H-S-ineq}
\left(\int_M \frac{|u|^{\crit}}{d_g(x,x_0)^s}\, dv_g\right)^{\frac{2}{\crit}}\leq K(n,s)\int_M |\nabla u|_g^2\, dv_g + B_0(M,g,s,x_0)\int_M u^2\, dv_g
\end{equation}
for all  $u \in H_1^2(M)$.
\end{theorem}

\medskip\noindent When $s=0$, Theorem \ref{Main-Th-Sharp-Cst-H-S-ineq} has been proved by Hebey-Vaugon \cite{Hebey-Vaugon} for best constant in the Sobolev embedding $H_1^2(M) \subset L^{2^\star}(M)$, $2^\star = 2n/(n-2)$. The best constant problem in the Sobolev embedding $H_1^p(M) \subset L^{p^\star}(M), p^\star = pn/(n-p)$ ($n>p>1$) has been studied by Druet \cite{Druet-1} (see also Aubin-Li \cite{aubinli}) answering a conjecture of Aubin in \cite{Aubin-1}. The corresponding question for the embedding $H_2^2(M) \subset L^{2^{\sharp}}(M)$, $2^{\sharp} = 2n/(n-4)$ has been studied by Hebey \cite{Hebey-1}, and for the Gagliardo-Nirenberg inequalities by Brouttelande \cite{brouttelande} and Ceccon-Montenegro \cite{cecconmontenegro}.





\medskip\noindent There is an important litterature about sharp constants for inequalities of Hardy-Sobolev type on domains of the Euclidean flat space $\mathbb{R}^n$. A general discussion is in the monograph \cite{ghoussoubmoradifam} by Ghoussoub-Moradifam. Hardy-Sobolev inequalities are a subfamily of the Caffarelli-Kohn-Nirenberg inequalities (see \cite{Caff-K-Nir}). The best constants and extremals for these inequalities on $\mathbb{R}^n$ are well understood in the class of radially symmetric functions (see Catrina-Wang \cite{catrinawang}, Horushi \cite{Horiushi} and Chou-Chu \cite{Chou-Chu}). However, there are situations when extremals are not radially symmetrical as discovered by Catrina-Wang \cite{catrinawang}. A historical survey on symmetry-breaking of the extremals is in Dolbeault-Esteban-Loss-Tarantello \cite{DELT}. For Hardy-Sobolev-Maz'ya inequalities \cite{Maz'ja}, we refer to Badiale-Tarantello \cite{badialetarantello}, Musina \cite{Musina} and Tertikas-Tintarev \cite{Tert-Tint}. 

\medskip\noindent A last remark is that it follows from the analysis of the author in \cite{jaber:test:fcts} that
\begin{equation}\label{min:B0}
\left\{\begin{array}{ll}
B_0(M,g,s,x_0) \geq K(n,s)\frac{(n-2)(6-s)}{12(2n-2-s)}\hbox{Scal}_g(x_0)&\hbox{ if }n\geq 4\\
\hbox{the Green's function's mass of }\Delta_g+ \frac{B_0(M,g,s,x_0)}{K(3,s)}\hbox{ is nonpositive}&\hbox{ if }n=3,
\end{array}\right.\end{equation}
where Scal$_g(x_0)$ is the scalar curvature at $x_0$. The mass is defined at the end of Section \ref{sec:proofth}.
 

\medskip\noindent The proof of Theorem \ref{Main-Th-Sharp-Cst-H-S-ineq} relies on the analysis of blowing-up families to critical nonlinear elliptic equations. In Section \ref{sec:blowup}, we prove a general convergence theorem for blowing-up solutions to Hardy-Sobolev equations. In Section \ref{sec:proofth},  we prove Theorem \ref{Main-Th-Sharp-Cst-H-S-ineq} by adapting the arguments of Druet (in Druet \cite{Druet-1}). We prove \eqref{min:B0} in Section \ref{sec:proofth}.


 
\section{Blow-up around $x_0$}\label{sec:blowup} We let $(M,g)$ be a smooth compact Riemannian Manifold of dimension $n \geq 3$, $x_0 \in M$, $s \in (0,2)$.
We consider a family $(u_\alpha)_{\alpha>0}$ in $H_1^2(M)$, 
such that for all $\alpha > 0$, $u_\alpha \geq 0$, $u_\alpha \not\equiv 0$ and $u_\alpha$ is a solution to the problem  
\begin{equation}\label{eqt-1-Th-2-Sharp-Cst-H-S-ineq}
\left\{
\begin{array}{ll}
\Delta_g u_\alpha + \alpha u_\alpha = \lambda_\alpha\frac{u_\alpha^{\crit-1}}{d_g(x,x_0)^s}\\
 \lambda_\alpha \in (0, K(n,s)^{-1}) \ , \ \|u_\alpha\|_{\crit,s}=1.
\end{array}\right.
\end{equation} 
Here $\Delta_g:=-\hbox{div}_g(\nabla)$ is the Laplace-Beltrami operator. It follows from the regularity and the maximum principle of \cite{jaber:test:fcts} that, for any $\alpha > 0$, $u_\alpha  \in C^{0,\beta}(M)\cap C_{loc}^{2,\gamma}(M\setminus\{x_0\})$, $\beta \in (0, \min(1,2-s))$, $\gamma \in (0,1)$, and $u_\alpha>0$. We define $I_\alpha$ as
\begin{equation}\label{eqt-1-Th-2-Sharp-Cst-H-S-ineq}
v \in H_1^2(M)\setminus\{0\} \mapsto I_\alpha(v) := \frac{\int_M |\nabla v|_g^2 dv_g + \alpha\int_Mv^2dv_g}{\|u_\alpha\|_{\crit,s}^2}.
\end{equation}
In particular, $I_\alpha(u_\alpha) = \int_M |\nabla u_\alpha|_g^2 dv_g + \alpha\int_Mu_\alpha^2dv_g = \lambda_\alpha$ for all $\alpha>0$.

\medskip\noindent We claim that
\begin{equation}\label{weak:lim}
u_\alpha \rightharpoonup 0\hbox{ weakly in }H_1^2(M)\hbox{ as }\alpha \to +\infty.
\end{equation}
We prove the claim. As one checked, $(u_\alpha)_{\alpha>0}$ is bounded in $H_1^2(M)$. Therefore, there exists $u^0 \in H_1^2(M)$ such that $u_\alpha \rightharpoonup u^0$ in  $H_1^2(M)$ as $\alpha \to +\infty$. By $I_\alpha(u_\alpha) = \lambda_\alpha$, 
we get that $\|u_\alpha\|_2 \leq C_1\alpha^{-1/2}$, where $C_1 > 0$ is independent of $\alpha$. Since $H_1^2(M)$ is compactly embedded in $L^2(M)$, we then get that $\|u^0\|_2 = 0$. Hence $u^0 \equiv 0$. This proves the claim.

\medskip\noindent Since $M$ is compact and $u_\alpha\in C^0(M)$ for all $\alpha > 0$, there exist $x_\alpha \in M$, $\mu_\alpha > 0$ such that 
\begin{equation}\label{eqt-2-Th-2-Sharp-Cst-H-S-ineq}
\max_M(u_\alpha) = u_\alpha(x_\alpha) = \mu_\alpha^{1-\frac{n}{2}}.
\end{equation}

\medskip\noindent In the sequel, we denote by $\B_\rho(z) \subset M$ the geodesic ball of radius $\rho$ centered at $z$.

\begin{proposition}\label{prop-1-Th-2-Sharp-Cst-H-S-ineq} We let $(u_\alpha)_{\alpha>0}$ be as in \eqref{eqt-1-Th-2-Sharp-Cst-H-S-ineq}. Then $u_\alpha \to 0$ as $\alpha \to +\infty$ in $C^0_{loc}(M\setminus\{x_0\})$.
\end{proposition}

\begin{proof} We consider $y \in M\setminus\{x_0\}$, $\rho_y = \frac{1}{3}d_g(y,x_0)$. By \eqref{eqt-1-Th-2-Sharp-Cst-H-S-ineq}, we have that $ \Delta_{g}u_\alpha \leq F_\alpha u_\alpha$ in $\B_{2\rho_y}(y)$, where $F_\alpha$ is the function defined by 
$ F_\alpha(x) = \lambda_\alpha u_\alpha^{\crit-2}/d_{g}(\cdot, x_0)^s$.  For any $r\in \left( \frac{n}{2}, \frac{n}{2-s} \right)$, we have that $\int_{\B_{2\rho_y}(y)}F_\alpha^rdv_{g} \leq C_2$  where $C_2 > 0$ is a constant independent of $\alpha$.
By Theorem 4.1 in Han-Lin \cite{Han-Lin}  (see also Lemma \ref{lemma-2-Appendix-Sharp-Cst-H-S-ineq} in the Appendix), 
it follows that there exists $C_3 = C_3(n,s,y,C_2) > 0$ independent of $\alpha$ such that, up to a subsequence, we have that 
$$\max_{\B_{\rho_y}(y)}u_\alpha \leq C_3\|u_\alpha\|_{L^2(\B_{2\rho_y})(y)}. $$ 
Since $u_\alpha \rightharpoonup 0$ in  $H_1^2(M)$ as $\alpha \to +\infty$ then by the last inequality, we get $$\lim_{\alpha\to +\infty}\|u_\alpha\|_{L^{\infty}(\B_{\rho_y}(y))} = 0.$$ 
Proposition \ref{prop-1-Th-2-Sharp-Cst-H-S-ineq} follows from a covering argument.\end{proof}



\begin{proposition}\label{prop-2-Th-2-Sharp-Cst-H-S-ineq} 
We let $(u_\alpha)_{\alpha>0}$ be as in \eqref{eqt-1-Th-2-Sharp-Cst-H-S-ineq}. Then $\sup_Mu_\alpha = +\infty$ as $\alpha \to +\infty$.
\end{proposition}

\begin{proof} We proceed by contradiction and assume that $\sup_Mu_\alpha  \not\rightarrow +\infty$ as $\alpha \to +\infty$. 
Then there exists $C_4 > 0$ independent of $\alpha$ such that $u_\alpha \leq C_4$. Since $u_\alpha \rightharpoonup 0$
as $\alpha \to +\infty$  in $H_1^2(M)$ then by dominated convergence Theorem we get that $\lim_{\alpha\to +\infty}\|u_\alpha\|_{\crit, s} = 0$. A contradiction since for all $\alpha > 0$, 
$\|u_\alpha\|_{\crit, s} = 1$. This ends the proof of Proposition \ref{prop-2-Th-2-Sharp-Cst-H-S-ineq}.
\end{proof} 
\noindent Propositions \ref{prop-1-Th-2-Sharp-Cst-H-S-ineq} and  \ref{prop-2-Th-2-Sharp-Cst-H-S-ineq} immediately yield the following:
\begin{proposition}\label{prop-3-Th-2-Sharp-Cst-H-S-ineq} 
We let $(u_\alpha)_{\alpha>0}$ be as in \eqref{eqt-1-Th-2-Sharp-Cst-H-S-ineq}. Then $x_\alpha \to x_0$ as $\alpha \to +\infty$.
\end{proposition}

\medskip\noindent In the sequel, we fix $R_0 \in (0, i_g(M))$, where $i_g(M)>0$ is the injectivity radius of $(M,g)$. We fix $\eta_0\in C^\infty_c(\B_{3R_0/4}(0)\subset\R^n)$ such that $\eta\equiv 1$ in $\B_{R_0/2}(0)$. The main result of this section is the following:
\begin{theorem}\label{Th2-Sharp-Cst-H-S-ineq} We let $(u_\alpha)_{\alpha>0}$ be as in \eqref{eqt-1-Th-2-Sharp-Cst-H-S-ineq}. We consider a sequence $(z_\alpha)_{\alpha>0}\in M$ such that $\lim_{\alpha\rightarrow+\infty}z_\alpha = x_0$.
We define the function $\hat{u}_\alpha$ on $\B_{R_0\mu_\alpha^{-1}}(0)\subset \mathbb{R}^n$ by
\begin{equation}
\hat{u}_\alpha(X) = \mu_\alpha^{\frac{n}{2}-1}u_\alpha(\exp_{z_\alpha}(\mu_\alpha X)),
\end{equation}
where $\exp_{z_\alpha}^{-1} : \Omega_\alpha \to \B_{R_0}(0)$ is the exponential map at $z_\alpha$.
We assume that 
\begin{equation}\label{hyp:dist}
d_g(x_\alpha, z_\alpha) = O(\mu_\alpha)\hbox{ when }\alpha\to +\infty.
\end{equation} 
Then 
\begin{equation}\label{concl:dist}
d_g(z_\alpha, x_0)=O(\mu_\alpha)\hbox{ when }\alpha\to +\infty
\end{equation}
and, up to a subsequence, $\eta_\alpha\hat{u}_\alpha$ converge to $\hat{u}$ weakly in $D_1^2(\R^n)$ and uniformly in $C_{loc}^{0,\beta}(\R^n)$, for all $\beta \in (0, \min(1,2-s))$, where 
$\eta_\alpha:=\eta_0(\mu_\alpha\cdot)$ and 
$$\hat{u}(X) = \left(\frac{a^{\frac{2-s}{2}}k^{\frac{2-s}{2}}}{a^{2-s} + |X-X_0|^{2-s}}\right)^{\frac{n-2}{2-s}}\hbox{ for all }X \in \R^n$$
with $X_0 \in \R^n$,  $a > 0$ and $k^{2-s} = (n-2)(n-s)K(n,s)$. In particular $\hat{u}$ verifies 
\begin{equation}\label{eqt-0-Th-2-Sharp-Cst-H-S-ineq}
 \Delta_\delta\hat{u} = K(n,s)^{-1}\frac{\hat{u}^{\crit-1}}{|X-X_0|^s} \hbox{ in }\R^n\ {\rm and} \ \int_{\R^n}\frac{\hat{u}^{\crit}}{|X-X_0|^s}dX = 1,
\end{equation}
where $|\cdot|$ is the Euclidean norm on $\R^n$ and $\delta$ is the Euclidean metric of $\R^n$. 
\end{theorem}

\subsection*{Proof of Theorem \ref{Th2-Sharp-Cst-H-S-ineq}} We consider $(u_\alpha)_\alpha$, $(z_\alpha)_{\alpha>0}\in M$ and $\hat{u}_\alpha$ as in the statement of the theorem. We define the metric $\hat{g}_\alpha : X \mapsto \exp_{z_\alpha}^*g(\mu_\alpha X)$ and also on $\R^n$, the vectors 
$X_\alpha = \mu_\alpha^{-1}\exp_{z_\alpha}^{-1}(x_\alpha)$ and $X_{0,\alpha} = \mu_\alpha^{-1}\exp_{z_\alpha}^{-1}(x_0)$. It follows from \eqref{hyp:dist} that
\begin{equation}\label{eqt-2.1-Th-2-Sharp-Cst-H-S-ineq}
\frac{d_g(z_\alpha, x_\alpha)}{\mu_\alpha} = |X_\alpha| = O(1)\hbox{ when }\alpha\to +\infty.
\end{equation} 
The proof of Theorem \ref{Th2-Sharp-Cst-H-S-ineq} proceeds in several steps :

\medskip\noindent{\bf Step 1.0 :} We claim that, for all $\alpha > 0$, $\hat{u}_\alpha$ verifies 
\begin{equation}\label{eqt-3.0-Th-2-Sharp-Cst-H-S-ineq}
\Delta_{\hat{g}_\alpha}\hat{u}_\alpha + \alpha\mu_\alpha^2\hat{u}_\alpha =
\lambda_\alpha\frac{\hat{u}_\alpha^{2^*(s)-1}}{d_{\hat{g}_\alpha}(X,X_{0,\alpha})^s}.
\end{equation}

\begin{proof} Indeed, we consider $\alpha > 0$, $X \in \B_{R_\alpha\mu_\alpha^{-1}}(0)$. Letting 
$x = \exp_{z_\alpha}(\mu_\alpha X)$, we then obtain
\begin{equation}\label{eqt-3.1-Th-2-Sharp-Cst-H-S-ineq}
\Delta_{\hat{g}_\alpha}\hat{u}_\alpha(X) =  \mu_\alpha^{{\frac{n}{2}+1}}\Delta_gu_\alpha(x).
\end{equation}
and 
\begin{equation}\label{eqt-3.2-Th-2-Sharp-Cst-H-S-ineq}
\frac{\hat{u}_\alpha^{2^*(s)-1}}{d_{\hat{g}_\alpha}(X,X_{0,\alpha})^s} =
\mu_\alpha^{({\frac{n}{2}-1})(2^*(s)-1)+s}\frac{u_\alpha^{2^*(s)-1}(x)}{d_g(x,x_0)^s}.
\end{equation}
Since $u_\alpha$ verifies equation \eqref{eqt-1-Th-2-Sharp-Cst-H-S-ineq} then 
plugging \eqref{eqt-3.1-Th-2-Sharp-Cst-H-S-ineq} and \eqref{eqt-3.2-Th-2-Sharp-Cst-H-S-ineq} into \eqref{eqt-1-Th-2-Sharp-Cst-H-S-ineq}, we get the claim. \end{proof}

\medskip\noindent{\bf Step 1.1:} We claim that there exists $\hat{u} \in D_1^2(\R^n)$, $\hat{u} \not\equiv 0$ such that, up to a subsequence, 
$(\eta_\alpha\hat{u}_\alpha)_{\alpha > 0}$ converge weakly to $\hat{u}$, as $\alpha \to +\infty$, in $D_1^2(\R^n)$. 

 \begin{proof} Indeed, for all $\alpha > 0$, we can write : 
\begin{equation}\label{eqt-4-Th-2-Sharp-Cst-H-S-ineq}
\int_{\R^n}|\nabla(\eta_\alpha\hat{u}_\alpha)|_{\hat{g}_\alpha}^2dv_{\hat{g}_\alpha} =
    \int_{\R^n}\eta_\alpha(\Delta_{\hat{g}_\alpha}\eta_\alpha) \hat{u}_\alpha^2dv_{\hat{g}_\alpha} +
    \int_{\R^n}\eta_\alpha^2|\nabla\hat{u}_\alpha|_{\hat{g}_\alpha}^2dv_{\hat{g}_\alpha}
\end{equation}
Since $\mu_\alpha \to 0$ as $\alpha \to +\infty$ then, up to a subsequence of $(\mu_\alpha)_{\alpha>0}$,  there exists $C_5 > 0$ independent of $\alpha$ such that we have in the sense of bilinear form 
\begin{equation}\label{eqt-5-Th-2-Sharp-Cst-H-S-ineq}
C_5^{-1}\delta(X) \leq \hat{g}_\alpha(X) \leq C_5\delta(X)
\end{equation}
for all $X \in \B_{\frac{3R_0}{4\mu_\alpha}}(0)$   where $\delta$ is the Euclidean metric on $\R^n$. Relations \eqref{eqt-4-Th-2-Sharp-Cst-H-S-ineq} and \eqref{eqt-5-Th-2-Sharp-Cst-H-S-ineq} imply that
there exists a constant $C_6 > 0$ independent of $\alpha$ such that 
\begin{equation}\label{eqt-6-Th-2-Sharp-Cst-H-S-ineq}
\int_{\R^n}\left|\eta_\alpha(\Delta_{\hat{g}_\alpha}\eta_\alpha)\right|\hat{u}_\alpha^2dv_{\hat{g}_\alpha} \leq C_6\mu_\alpha^2\int_{\B_{\frac{3R_0}{4\mu_\alpha}}(0)}\hat{u}_\alpha^2dv_{\hat{g}_\alpha},
\end{equation}
By passing to $\B_{\frac{3R_0}{4}}(z_\alpha)$ in \eqref{eqt-6-Th-2-Sharp-Cst-H-S-ineq} via the exponential chart $(\Omega_\alpha, \exp_{z_\alpha}^{-1})$ (by taking $x = \exp_{z_\alpha}(\mu_\alpha X)$), we obtain  that
\begin{equation}\label{eqt-7-Th-2-Sharp-Cst-H-S-ineq}
\int_{\R^n}\left|\eta_\alpha(\Delta_{\hat{g}_\alpha}\eta_\alpha)\right|\hat{u}_\alpha^2dv_{\hat{g}_\alpha} \leq C_7\int_Mu_\alpha^2dv_g,
\end{equation}
where $C_7 > 0$ is a constant independent of $\alpha$. Since $\Vert u_\alpha\Vert_{H_1^2(M)}= O(1)$ when $\alpha\to +\infty$, then relation \eqref{eqt-7-Th-2-Sharp-Cst-H-S-ineq} yields
\begin{equation}\label{eqt-8-Th-2-Sharp-Cst-H-S-ineq}
\int_{\R^n}|\nabla\eta_\alpha|_{\hat{g}_\alpha}^2\hat{u}_\alpha^2dv_{\hat{g}_\alpha} \leq C_8,
\end{equation}where $C_8 > 0$ is a constant independent of $\alpha$.
On the other hand, we write that 
\begin{eqnarray}\label{eqt-9-Th-2-Sharp-Cst-H-S-ineq}
 \int_{\R^n}\eta_\alpha^2|\nabla\hat{u}_\alpha|_{\hat{g}_\alpha}^2dv_{\hat{g}_\alpha} &\leq&  \int_{\B_{\frac{3R_0}{4}}(z_\alpha)}|\nabla u_\alpha|_g^2dv_g \nonumber
\\ &\leq& I_\alpha(u_\alpha) - \int_M\alpha u_\alpha^2dv_g<K(n,s)^{-1}
\end{eqnarray}
for all $\alpha > 0$. Plugging \eqref{eqt-8-Th-2-Sharp-Cst-H-S-ineq} and \eqref{eqt-9-Th-2-Sharp-Cst-H-S-ineq} into \eqref{eqt-4-Th-2-Sharp-Cst-H-S-ineq}, we then obtain that  
$$\int_{\R^n}|\nabla(\eta_\alpha\hat{u}_\alpha)|_{\hat{g}_\alpha}^2dv_{\hat{g}_\alpha} \leq C_9 $$
where $C_9 > 0$ is a constant independent of $\alpha$. The last relation and \eqref{eqt-5-Th-2-Sharp-Cst-H-S-ineq} give 
\begin{eqnarray}\label{eqt-10-Th-2-Sharp-Cst-H-S-ineq}
\int_{\R^n}|\nabla(\eta_\alpha\hat{u}_\alpha)|_\delta^2dX &\leq& C_5^{n/2+1}\int_{\R^n}|\nabla(\eta_\alpha\hat{u}_\alpha)|_{\hat{g}_\alpha}^2dv_{\hat{g}_\alpha} \leq C_{10}
\end{eqnarray}
where $C_{10} > 0$ is a constant independent of $\alpha$. This implies that the sequence $(\eta_\alpha\hat{u}_\alpha)_{\alpha > 0}$ 
is bounded in $D_1^2(\R^n)$ then there exists $\hat{u} \in D_1^2(\R^n)$ such that, up to a subsequence, $\eta_\alpha\hat{u}_\alpha \rightharpoonup \hat{u}$ as $\alpha \to +\infty$. 

\medskip\noindent It remains to prove that $\hat{u} \not\equiv 0$. Indeed, since $\hat{u}_\alpha \leq 1$ and $\lambda_\alpha \in (0, K(n,s)^{-1})$ then for all $X \in \B_{R_\alpha\mu_\alpha^{-1}}(0)$, we can write :
\begin{eqnarray}\label{eqt-11-Th-2-Sharp-Cst-H-S-ineq}
    \Delta_{\hat{g}_\alpha}\hat{u}_\alpha(X) &=&
    \lambda_\alpha\frac{\hat{u}_\alpha^{2^*(s)-1}}{d_{\hat{g}_\alpha}(X,X_{0,\alpha})^s} - \alpha\mu_\alpha^2\hat{u}_\alpha \nonumber
    \\ &\leq& \frac{K(n,s)^{-1}}{d_{\hat{g}_\alpha}(X,X_{0,\alpha})^s}\hat{u}_\alpha = F_\alpha(X) \hat{u}_\alpha,
\end{eqnarray}
where
$$ F_\alpha(X):=\frac{K(n,s)^{-1}}{d_{\hat{g}_\alpha}(X,X_{0,\alpha})^s}.$$
We consider $r \in (\frac{n}{2}, \frac{n}{s})$. It follows from \eqref{eqt-5-Th-2-Sharp-Cst-H-S-ineq} that
\begin{equation}\label{eqt-12-Th-2-Sharp-Cst-H-S-ineq}C_5^{-1/2}|X - X_{0,\alpha}| \leq d_{\hat{g}_\alpha}(X,X_{0,\alpha}) \leq C_5^{1/2}|X - X_{0,\alpha}|.
\end{equation}
and
\begin{equation}\label{eqt-13-Th-2-Sharp-Cst-H-S-ineq}
 C_5^{-n/2} \leq \sqrt{det(\hat{g}_\alpha)}(X) \leq C_5^{n/2}
\end{equation}
for all $X \in \B_{R_0}(0)$. We distinguish two cases : 

\medskip\noindent{\it Case 1.1.1 :} $X_{0,\alpha} \to X_0$ as $\alpha \to +\infty$. In this case, we get with \eqref{eqt-2.1-Th-2-Sharp-Cst-H-S-ineq} that for all $\alpha > 0$, $X_\alpha, X_{0,\alpha} \in \B_{R_1}(0)$, 
for $R_1 > 0$ sufficiently large
and by \eqref{eqt-12-Th-2-Sharp-Cst-H-S-ineq} and \eqref{eqt-13-Th-2-Sharp-Cst-H-S-ineq}, we obtain that $\int_{\B_{2R_1}(0)}F_\alpha^rdv_{\hat{g}_\alpha} \leq C_{11}, $
where  $C_{11} > 0$ is a constant independent of $\alpha$.

\medskip\noindent{\it Case 1.1.2 :} $X_{0,\alpha} \to +\infty$ as $\alpha \to +\infty$. In this case, coming back to relations \eqref{eqt-12-Th-2-Sharp-Cst-H-S-ineq}, \eqref{eqt-13-Th-2-Sharp-Cst-H-S-ineq} and
by dominated convergence Theorem,  we get that $\lim_{\alpha\rightarrow+\infty}\|F_\alpha\|_{C^0(\B_{2R_1}(0))} = 0$, where for all $\alpha > 0$, $X_\alpha \in \B_{R_1}(0)$.
It follows that $\int_{\B_{R_1}(0)}F_\alpha^rdv_{\hat{g}_\alpha} \leq C_{12}$ where $C_{12} > 0$  is independent of $\alpha$.

\medskip\noindent Hence, we get in both cases, up to a subsequence,  that there exists $R_1 > 0$ such that  for all $\alpha > 0$, $X_\alpha \in \B_{R_1}(0)$ and $\int_{\B_{2R_1}(0)}F_\alpha^rdv_{\hat{g}_\alpha} \leq C_{13}, $
where $C_{13} = \max(C_{11}, C_{12})$. Moreover,  for all $\theta \leq \min(1, 2-s)$,  $\hat{u}_\alpha \in C^{0,\theta}(\B_{2R_1}(0))$ and
$\hat{u}_\alpha> 0$. Thanks to Theorem 4.1 in Han-Lin \cite{Han-Lin}  (see also Lemma \ref{lemma-1-Appendix-Sharp-Cst-H-S-ineq} in the Appendix), we get that 
\begin{equation}\label{eqt-14.0-Th-2-Sharp-Cst-H-S-ineq}
1= \max_{\B_{R_1}(0)}\hat{u}_\alpha \leq C_{14}\|\hat{u}_\alpha\|_{L^2(\B_{2R_1}(0))}, 
\end{equation}
where  $C_{14} > 0$ is a constant independent of $\alpha$. 
Since $(\eta_\alpha\hat{u}_\alpha)_{\alpha > 0}$ is bounded and converges weakly to $\hat{u}$ as  $\alpha \to +\infty$ in $D_1^2(\R^n)$, the convergence is strong in $L^2_{loc}$ and then, letting $\alpha \to +\infty$ in \eqref{eqt-14.0-Th-2-Sharp-Cst-H-S-ineq},
we get that $\|\hat{u}\|_{L^2(\B_{2R_1}(0))} \geq C_{14}^{-1}$ and then $\hat{u} \not\equiv 0$. This ends the proof of Step 1.1.\end{proof}

\medskip\noindent{\bf Step 1.2:} We claim that $\lambda_\alpha \to K(n,s)^{-1}$ as $\alpha \to +\infty$. 

\begin{proof} Indeed, since for all $\alpha > 0$, we have $\lambda_\alpha \in (0, K(n,s)^{-1})$ 
then, up to a subsequence, $\lambda_\alpha \to \lambda\leq K(n,s)^{-1}$ as $\alpha \to +\infty$. We proceed by contradiction and assume that $\lambda \neq K(n,s)^{-1}$. Then there exist $\epsilon_0 > 0$ 
and $\alpha_0 > 0$ such that for all $\alpha > \alpha_0$ :    
\begin{equation}\label{eqt-14-Th-2-Sharp-Cst-H-S-ineq}
K(n,s)^{-1} > \lambda + \epsilon_0.
\end{equation}
By Jaber \cite{jaber:test:fcts},  for all $\epsilon > 0$ there exist $B_\epsilon > 0$ such that for all $\alpha > 0$, we have : 
\begin{equation}\label{eqt-15-Th-2-Sharp-Cst-H-S-ineq}
\left(\int_M\frac{|u_\alpha|^{2^\star(s)}}{d_g(x,x_0)^s}dv_g\right)^{\frac{2}{2^\star(s)}}
\leq (K(n,s) + \epsilon)\int_M |\nabla u_\alpha|_g^2 dv_g + B_\epsilon \int_Mu_\alpha^2dv_g.
\end{equation}
Since $\|u_\alpha\|_{\crit, s} = 1, I_\alpha(u_\alpha) = \lambda_\alpha$ and $u_\alpha \to 0$ in $L^2(M)$ as $\alpha \to +\infty$ then $$1 \leq  \left(\frac{1}{\lambda_\alpha + \epsilon_0} + \epsilon\right)\lambda_\alpha+o(1). $$
Letting $\alpha \to +\infty$ and then  $\epsilon \to 0$ in the last relation, we obtain that $\frac{\lambda}{\lambda + \epsilon_0} \geq 1$, a contradiction since $\lambda \geq 0$ and $\epsilon_0 > 0$.  
This proves that $\lambda = K(n,s)^{-1}$.\end{proof}

\medskip\noindent{\bf Step 1.3:} We claim that there exists $A \geq 0$ such that $\hat{u}$ verifies on $C_c^\infty(\R^n)$ : 
\begin{equation}\label{eqt-16-Th-2-Sharp-Cst-H-S-ineq}
\Delta_\delta\hat{u} + A\hat{u} = \left\{
\begin{array}{lll}
K(n,s)^{-1}\frac{\hat{u}^{2^*(s)-1}}{|X-X_0|^s} & {\rm if} \;
X_{0,\alpha} \xrightarrow{\alpha\rightarrow+\infty} X_0 \hbox{}
\\ 0     & {\rm if} \; 
|X_{0,\alpha}| \xrightarrow{\alpha\rightarrow+\infty} +\infty. \hbox{}
\end{array}\right.
\end{equation}

\begin{proof} We consider $R > 0$ and $\varphi \in C^\infty(\B_R(0))$. Indeed, thanks to Cartan's expansion of the metric $g$ (see for instance \cite{Lee-parker}), we have for all $\alpha > 0$ : 
$$\hat{g}_\alpha(X) =  \delta(z_\alpha) + o(\mu_\alpha) $$ 
uniformly on $\B_R(0)$. This implies that 
\begin{equation}\label{eqt-17-Th-2-Sharp-Cst-H-S-ineq}
\int_{\R^n}<\nabla\hat{u}_\alpha, \nabla\varphi>_{\hat{g}_\alpha}dv_{\hat{g}_\alpha}
= \int_{\B_R(0)}<\nabla\hat{u}_\alpha, \nabla\varphi>_\delta dX + o(\mu_\alpha)
\end{equation}
Since $\eta_\alpha\hat{u}_\alpha  \rightharpoonup \hat{u}$ on $D_1^2(\R^n)$ and $\mu_\alpha \to 0$ as $\alpha \to +\infty$ then by \eqref{eqt-17-Th-2-Sharp-Cst-H-S-ineq}, we get that 
\begin{equation}\label{eqt-18-Th-2-Sharp-Cst-H-S-ineq}
\lim_{\alpha\to +\infty}\int_{\R^n}<\nabla(\eta_\alpha\hat{u}_\alpha), \nabla\varphi>_{\hat{g}_\alpha}dv_{\hat{g}_\alpha}
= \int_{\B_R(0)}<\nabla\hat{u}, \nabla\varphi>_\delta dX.
\end{equation}
Now, since $I_\alpha(u_\alpha) = \lambda_\alpha$ and $\lambda_\alpha \in (0, K(n,s)^{-1})$ then we get 
\begin{equation}\label{eqt-19-Th-2-Sharp-Cst-H-S-ineq}
\alpha\mu_\alpha^2\int_{\B_R(0)}\hat{u}_\alpha^2dv_{\hat{g}_\alpha}  < K(n,s)^{-1}.
\end{equation}
By dominated convergence Theorem, we obtain that 
\begin{equation}\label{eqt-20-Th-2-Sharp-Cst-H-S-ineq}
\int_{\B_R(0)}\hat{u}_\alpha^2dv_{\hat{g}_\alpha} \xrightarrow{\alpha\rightarrow+\infty} \int_{\B_R(0)}\hat{u}^2dX.
\end{equation}
Together, Relations \eqref{eqt-19-Th-2-Sharp-Cst-H-S-ineq} and \eqref{eqt-20-Th-2-Sharp-Cst-H-S-ineq} give that 
$$\alpha\mu_\alpha^2 \leq \frac{K(n,s)^{-1}}{\int_{\B_R(0)}\hat{u}^2dX} + o(1).  $$
Hence, $\alpha\mu_\alpha^2 = O(1)$ and there exists $A \geq 0$ such that, up to a subsequence, $\lim_{\alpha\rightarrow+\infty}\alpha\mu_\alpha^2 = A.$  
Using dominated convergence Theorem again, we obtain that 
\begin{equation}\label{eqt-21-Th-2-Sharp-Cst-H-S-ineq}
\lim_{\alpha\rightarrow+\infty}\int_{\R^n}\alpha\mu_\alpha^2\hat{u}_\alpha\varphi dv_{\hat{g}_\alpha} = A\int_{\R^n}\hat{u}\varphi dX.
\end{equation}  
At last, we consider the sequence $(h_\alpha)_{\alpha>0}$ defined on $\B_R(0)$ by :
$$X \in \B_R(0) \mapsto h_\alpha(X) =
\frac{\hat{u}_\alpha^{2^*(s)-1}\varphi}{d_{\hat{g}_\alpha}(X,X_{0,\alpha})^s}\sqrt{det(\hat{g}_\alpha)}.$$
We claim that $$\lim_{\alpha\rightarrow+\infty}\int_{\B_R(0)}h_\alpha dX =  \left\{
\begin{array}{lll}
0 & {\rm if} \; |X_{0,\alpha}| \xrightarrow{\alpha\rightarrow+\infty} +\infty \hbox{}
\\ \int_{\B_R(0)}\frac{\varphi\hat{u}^{2^*(s)-1}dX}{|X-X_0|^s}    & {\rm if} \; 
X_{0,\alpha} \xrightarrow{\alpha\rightarrow+\infty} X_0. \hbox{}
\end{array}\right.$$
We distinguish two cases : 

\medskip\noindent{\it Case 1.3.1 :} $X_{0,\alpha} \to +\infty$ as $\alpha \to +\infty$. In this case, $\lim_{\alpha\rightarrow+\infty}d_{\hat{g}_\alpha}(X,X_{0,\alpha})^{-s} = 0$ in $C_c^0(\R^n)$. Hence 
$\lim_{\alpha\rightarrow+\infty}\int_{\B_R(0)}h_\alpha dX = 0.$
This proves the claim in case 1.3.1.

\medskip\noindent{\it Case 1.3.2 :} There exists $X_0 \in \R^n$ such that $X_{0,\alpha} \to X_0$ as $\alpha \to +\infty$. 
Let us consider  the function $h$ defined on $\B_R(0)$ by $X \mapsto h(X) = (\hat{u}^{2^*(s)-1}\varphi)(X)/|X-X_0|_\delta^s$. 
In this case, for all $\epsilon > 0$, there exists 
$\alpha_1 = \alpha_1(\epsilon) > 0$ such that for all $\alpha > \alpha_1$, $X_{0,\alpha} \in  \B_{\frac{\epsilon}{2}}(X_0)$. Then for all $X \in \B_R(0)\setminus\B_{\epsilon}(X_0)$, we have :
$|X-X_{0,\alpha}|_\delta \geq \frac{\epsilon}{2}$. This implies that there exists  a constant $C_{15} = C_{15}(\epsilon) > 0$  independent of $\alpha$ such that $|h_\alpha| \leq C_{15}\cdot|\varphi|$.
Coming back to dominated convergence Theorem, we obtain with the last relation that 
\begin{equation}\label{eqt-23-Th-2-Sharp-Cst-H-S-ineq}
\lim_{\alpha\rightarrow+\infty}\int_{\B_R(0)\setminus\B_\epsilon(X_0)}h_\alpha(X)dX
= \int_{\B_R(0)\setminus\B_\epsilon(X_0)}h(X)dX.
\end{equation}
On the other hand, we get by \eqref{eqt-12-Th-2-Sharp-Cst-H-S-ineq} and \eqref{eqt-13-Th-2-Sharp-Cst-H-S-ineq} that 
\begin{eqnarray} \label{eqt-24-Th-2-Sharp-Cst-H-S-ineq}
\left|\int_{\B_\epsilon(X_0)}h_\alpha dX\right| &\leq&  C_{10}\|\varphi\|_\infty\int_{\B_{2\epsilon}(X_{0,\alpha})}\frac{dX}{|X-X_{0,\alpha}|^s}\leq C_{16}\cdot\epsilon^{n-s}.
\end{eqnarray}
where $C_{16} > 0$ is a constant independent of $\alpha$.
In a similar way,  we prove that 
\begin{equation} \label{eqt-25-Th-2-Sharp-Cst-H-S-ineq}
\left|\int_{\B_\epsilon(X_0)}h dX\right| \leq C_{17}\epsilon^{n-s},
\end{equation}where $C_{17} > 0$ is a constant independent of $\alpha$.
Combining \eqref{eqt-23-Th-2-Sharp-Cst-H-S-ineq}, \eqref{eqt-24-Th-2-Sharp-Cst-H-S-ineq}  and \eqref{eqt-25-Th-2-Sharp-Cst-H-S-ineq}, it follows for all $\alpha > \alpha_1$ that 
$$\left|\int_{\B_R(0)}h_\alpha dX - \int_{\B_R(0)}h dX\right| = o_\alpha(1) + O(\epsilon^{n-s}).$$ 
Letting $\alpha \to +\infty$ then $\epsilon \to 0$ in the last relation, we obtain that 
\begin{equation}\label{eqt-26.0-Th-2-Sharp-Cst-H-S-ineq}
\lim_{\alpha\rightarrow+\infty}\int_{\R^n}\frac{\hat{u}_\alpha^{2^*(s)-1}\varphi}{d_{\hat{g}_\alpha}(X, X_0)^s} dv_{\hat{g}_\alpha} = \int_{\R^n}\frac{\hat{u}^{2^*(s)-1}\varphi}{|X-X_0|^s}dX.
\end{equation}This proves the claim in case 1.3.2.
Hence, by combining relations \eqref{eqt-18-Th-2-Sharp-Cst-H-S-ineq}, \eqref{eqt-21-Th-2-Sharp-Cst-H-S-ineq}  and \eqref{eqt-26.0-Th-2-Sharp-Cst-H-S-ineq} with \eqref{eqt-3.0-Th-2-Sharp-Cst-H-S-ineq},
we get \eqref{eqt-16-Th-2-Sharp-Cst-H-S-ineq}. This ends Step 1.3.\end{proof}

\medskip\noindent{\bf Step 1.4:} We claim that $X_{0,\alpha} = \mu_\alpha^{-1}\exp_{z_\alpha}^{-1}(x_0) $ is bounded when $\alpha\to +\infty$. 

\begin{proof} We proceed by contradiction and we assume that $|X_{0,\alpha}| \to +\infty$ as $\alpha \to +\infty$. We proved in Step 1.3 that we obtain  in this case :   
\begin{equation}\label{eqt-26-Th-2-Sharp-Cst-H-S-ineq}
\Delta_\delta\hat{u} + A\hat{u} = 0,
\end{equation}
weakly on $C_c^\infty(\R^n)$. Let $\hat{\eta} \in C^\infty(\R^n)$ be such that $\hat{\eta} \equiv 1$ in $\B_1(0)$, $0 \leq \hat{\eta} \leq 1$ and $\hat{\eta} \equiv 0$ in $\R^n\setminus\B_2(0)$.
Now, we consider $R > 0$ and  define the function $\hat{\eta}_R$ on $\R^n$ by $\hat{\eta}_R(X) = \eta(R^{-1} X)$. Multiplying \eqref{eqt-26-Th-2-Sharp-Cst-H-S-ineq} by $\hat{\eta}_R\hat{u}$ and 
integrating by parts, we get that
\begin{equation} \label{eqt-28-Th-2-Sharp-Cst-H-S-ineq}
 \int_{\R^n}(\nabla\hat{u}, \nabla(\hat{\eta}_R\hat{u}))_\delta dX + A\int_{\R^n}\hat{\eta}_R\hat{u}^2 = 0.
\end{equation}
To get the contradiction, we need the following lemma : 

\begin{lemma}\label{lemma-2-Th-2-Sharp-Cst-H-S-ineq} We claim that 
 $$\lim_{R\rightarrow+\infty}\int_{\R^n}(\nabla\hat{u}, \nabla(\hat{\eta}_R\hat{u}))_\delta dX = \|\hat{u}\|_{D_1^2(\R^n)}^2. $$
\end{lemma}

\medskip\noindent{\it Proof of Lemma \ref{lemma-2-Th-2-Sharp-Cst-H-S-ineq}:} 
 Indeed, we have that : 
\begin{equation}\label{eqt-1-lemma-1-Th-2-Sharp-Cst-H-S-ineq}
 \int_{\R^n}(\nabla\hat{u}, \nabla(\hat{\eta}_R\hat{u}))_\delta dX = \int_{\R^n}\hat{\eta}_R|\nabla\hat{u}|_\delta^2dX + \int_{\R^n}(\nabla\hat{u},\nabla(\hat{\eta}_R)\hat{u})_\delta dX.
\end{equation}
Applying dominated convergence Theorem, we get that  
\begin{equation}\label{eqt-2-lemma-1-Th-2-Sharp-Cst-H-S-ineq}
 \lim_{R\rightarrow+\infty}\int_{\R^n}\hat{\eta}_R|\nabla\hat{u}|_\delta^2dX = \|\hat{u}\|_{D_1^2(\R^n)}^2.
\end{equation}
On the other hand, we obtain by Inequalities of Cauchy-Schwarz  then by H\"older's inequalities that 
\begin{eqnarray*}
 \left |\int_{\R^n}(\nabla\hat{u},\nabla(\hat{\eta}_R)\hat{u})_\delta dX \right| &\leq&
\|\hat{u}\|_{D_1^2(\R^n)}^2\times\frac{C_{18}}{R^2}\int_{\B_{2R}(0)\setminus\B_R(0)}\hat{u}^2dX
\\ &\leq& C_{19}\|\hat{u}\|_{D_1^2(\R^n)}^2\times
\left(\int_{\B_{2R}(0)\setminus\B_R(0)}\hat{u}^{2^*}dX\right)^{\frac{2}{2^*}},
\end{eqnarray*}
where $2^*=2n/(n-2)$ and $C_{18}, C_{19} > 0$ are independents of $\alpha$. It follows from the last relation, Sobolev's embedding theorem and the dominated convergence theorem that  
\begin{equation}\label{eqt-3-lemma-1-Th-2-Sharp-Cst-H-S-ineq}
\lim_{R\to +\infty}\int_{\R^n}(\nabla\hat{u},\nabla(\hat{\eta}_R)\hat{u})_\delta dX  = 0.
\end{equation}
Letting $R \to +\infty$ in \eqref{eqt-1-lemma-1-Th-2-Sharp-Cst-H-S-ineq} and thanks to relations \eqref{eqt-2-lemma-1-Th-2-Sharp-Cst-H-S-ineq} and \eqref{eqt-3-lemma-1-Th-2-Sharp-Cst-H-S-ineq}, we get the claim. This proves Lemma \ref{lemma-2-Th-2-Sharp-Cst-H-S-ineq}.

\medskip\noindent Now, going back to relation \eqref{eqt-28-Th-2-Sharp-Cst-H-S-ineq} and thanks to Lemma \ref{lemma-2-Th-2-Sharp-Cst-H-S-ineq}, we get that 
$$\|\hat{u}\|_{D_1^2(\R^n)}^2 + A\int_{\R^n}\hat{\eta}_R\hat{u}^2 = o_R(1),$$
where $\lim_{R\to +\infty}o_R(1)=0$. Thus is a contradiction since $\hat{\eta}_R\hat{u}^2 \geq 0$ and $\hat{u} \not\equiv 0$.  This contradiction completes the proof of Step 1.4.\end{proof}
\medskip\noindent As a consequence, Step 1.4 implies that $|X_{0,\alpha}|=O(1)$ when $\alpha\to +\infty$, which yields \eqref{concl:dist}. Therefore, there exists $X_0 \in \R^n$ such that the function $\hat{u}$ verifies in the distribution sense : 
\begin{equation} \label{eqt-29-Th-2-Sharp-Cst-H-S-ineq}
 \Delta_\delta\hat{u} + A\hat{u} = K(n,s)^{-1}\frac{\hat{u}^{2^*(s)-1}}{|X-X_0|^s}.
\end{equation}

\medskip\noindent{\bf Step 1.5:} We claim that $A = 0$. 

\begin{proof} We proceed by contradiction and assume that $A > 0$. At first, let us prove that $\hat{u} \in L^2(\R^n)$. Multiplying \eqref{eqt-29-Th-2-Sharp-Cst-H-S-ineq} by $\hat{\eta}_R\hat{u}$ and integrating over $\R^n$, we obtain  
\begin{equation}\label{eq:1}
\int_{\R^n}(\nabla\hat{u}, \nabla(\hat{\eta}_R\hat{u}))_\delta dX + A\int_{\R^n}\hat{\eta}_R\hat{u}^2dX = K(n,s)^{-1}  \int_{\R^n}\hat{\eta}_R\frac{\hat{u}^{2^*(s)}}{|X-X_0|^s}\, dX.
\end{equation}

\medskip\noindent We claim that $\hat{u}^{2^*(s)}|X-X_0|^{-s}\in L^1(\R^n)$. We prove the claim. For all $\alpha > 0$, we have that $\int_M\frac{u_\alpha^{2^*(s)}}{d_g(x,x_0)^s}dv_g = 1$. Then for $R > 0$, we obtain by a change of variable that $\int_{\B_R(0)}\frac{|\eta_\alpha\hat{u}_\alpha|^{2^*(s)}}{d_{\hat{g}_\alpha}(X,X_{0,\alpha})^s}dv_{\hat{g}_\alpha} \leq 1$. 
Letting $\alpha \to +\infty$ then $R \to +\infty$, we get that 
\begin{equation} \label{eqt-32-Th-2-Sharp-Cst-H-S-ineq}
 \int_{\R^n}\frac{\hat{u}^{2^*(s)}}{|X-X_0|^s}\, dX \leq 1.
\end{equation}
This proves the claim.

\smallskip\noindent Letting $R \to +\infty$ in \eqref{eq:1} and using \eqref{eqt-32-Th-2-Sharp-Cst-H-S-ineq}, we get, thanks to Lemma \ref{lemma-2-Th-2-Sharp-Cst-H-S-ineq}, that $\lim_{R\to +\infty}A\int_{\R^n}\hat{\eta}_R\hat{u}^2dX \leq C_{20}$, where
$C_{20} > 0$ is independent of $\alpha$. Applying Beppo-Livi Theorem 
in the last relation, we get that $\hat{u}^2 \in L^1(\R^n)$. 
\\ Now, we consider the function 
\begin{eqnarray*}
f: \R^n\setminus\{X_0\}\times\R &\rightarrow& \R 
\\ (X,v) &\mapsto& f(X,v) = \left(\frac{K(n,s)^{-1}|v|^{2^*(s)-2}}{|X-X_0|_\delta^s}-A\right)v.
\end{eqnarray*}
$f$ is clearly continuous on $\R^n\setminus\{X_0\}\times\R$ and $\hat{u}$ verifies $\Delta_\delta\hat{u} = f(X,\hat{u})$, it follows by standard elliptic theory (see for instance \cite{Gil-Tru}) that 
$\hat{u} \in  C_{c}^\infty(\R^n\setminus\{X_0\}) \cap H_{2,loc}^1(\R^n\setminus\{X_0\}) $. By the Claim 5.3 in \cite{Rob-Pu-F.Rob}
 (once one checks that  $\hat{u}$ and $f$ satisfy all the condition of the Claim), we obtain after simple computations that $A\int_{\R^n}\hat{u}^2dX =0$. A contradiction since 
$\hat{u} \in L^2(\R^n)$ and $\hat{u} \not\equiv 0$. This ends the proof of Step 1.5.\end{proof}

\medskip\noindent As a consequence, Step 1.6 implies that there exists $X_0 \in \R^n$ such that the function $\hat{u}$ verifies in the distribution sense  : 
\begin{equation} \label{eqt-30-Th-2-Sharp-Cst-H-S-ineq}
 \Delta_\delta\hat{u} = K(n,s)^{-1}\frac{\hat{u}^{\crit-1}}{|X-X_0|^s}.
\end{equation}

\medskip\noindent{\bf Step 1.6 :} We claim that there exists $a >0$ such that 
\begin{equation}\label{eq:h:u}
\hat{u}(X) = \left(\frac{a^{\frac{2-s}{2}}k^{\frac{2-s}{2}}}{a^{2-s} + |X-X_0|^{2-s}}\right)^{\frac{n-2}{2-s}}\hbox{ for all }X \in \R^n,
\end{equation}
where $k^{2-s}:=(n-s)(n-2)K(n,s)$.

\begin{proof} Indeed,  Multiplying \eqref{eqt-30-Th-2-Sharp-Cst-H-S-ineq} by $\hat{\eta}_R\hat{u}$, integrating over $\R^n$ and letting $R\to +\infty$, we obtain  that
$$\int_{\R^n}|\nabla\hat{u}|^2dX = K(n,s)^{-1}  \int_{\R^n}\frac{\hat{u}^{2^*(s)}}{|X-X_0|^s}.$$
Thanks to the definition of $K(n,s)$ and with the last relation, we get that  
\begin{equation} \label{eqt-31-Th-2-Sharp-Cst-H-S-ineq}
 \int_{\R^n}\frac{\hat{u}^{2^*(s)}}{|X-X_0|^s} \geq 1.
\end{equation}
Inequalities \eqref{eqt-31-Th-2-Sharp-Cst-H-S-ineq} and \eqref{eqt-32-Th-2-Sharp-Cst-H-S-ineq} give that $$ \int_{\R^n}\frac{\hat{u}^{2^*(s)}}{|X-X_0|^s} = 1. $$
This implies that, up to a translation, $\hat{u}$ is a minimizer for the Euclidean Hardy-Sobolev inequality. By Lemma 3 in \cite{DELT} (see Chou-Chu \cite{Chou-Chu}, Horiuchi \cite{Horiushi} and also Theorem 4.3 in Lieb \cite{Lieb-1} and Theorem 4 in Catrina-Wang \cite{catrinawang}), 
we get that $\hat{u}(X) = b\left(c + |X-X_0|^{2-s}\right)^{-\frac{n-2}{2-s}}$ for some $b\neq 0$ and $c>0$. Since $\hat{u}$ satisfies \eqref{eqt-30-Th-2-Sharp-Cst-H-S-ineq}, we get \eqref{eq:h:u}. This proves the claim.\end{proof}    

\medskip\noindent{\bf Step 1.7 :} We claim that, up to a subsequence, $\eta_\alpha\hat{u}_\alpha \to \hat{u}$ in $C_{loc}^{0,\beta}(\R^n)$, for all $\beta \in (0, \min(1,2-s))$,

\begin{proof} Given $R' > 0$, we get by Step 1.0 (equation \eqref{eqt-3.0-Th-2-Sharp-Cst-H-S-ineq}) and Step 1.6, up to a subsequence of $(\hat{u}_\alpha)_{\alpha>0}$, that 
$\Delta_{\hat{g}_\alpha}\hat{u}_\alpha = F_\alpha$ on $C^\infty(\B_R(0))$  where $F_\alpha(X) =  - \alpha\mu_\alpha^2\hat{u}_\alpha + \lambda_\alpha\frac{\hat{u}_\alpha^{2^*(s)-1}}{d_{\hat{g}_\alpha}(X,X_{0,\alpha})^s}$. 
We consider $p \in (n/2, \inf(n/s, n))$. Knowing that $\hat{u}_\alpha \leq 1$ leads to $\Vert F_\alpha \Vert_{ L^p(\B_{R'}(0))}=O(1)$. 

\medskip\noindent It follows by standard elliptic theory (see \cite{Gil-Tru})
that for all $\beta \in (0, \min(1,2-s))$ and all $R < R'$, $(\hat{u}_\alpha) \in C^{0,\beta}(\B_R(0))$ and there exists $C_{21} = C_{21}(M,g,s,R,R',\beta) > 0$  such that $\|\hat{u}_\alpha\|_{C^{0,\beta}(\B_R(0))} \leq C$. 
Therefore the convergence holds in $C^{0,\beta'}(\B_R(0))$ for all $\beta'<\beta$. This ends the proof of Step 1.7.\end{proof}

\medskip\noindent Theorem \ref{Th2-Sharp-Cst-H-S-ineq} follows from Steps 1.0 to 1.7.

\medskip
\begin{corollary}\label{Cor-Th2-Sharp-Cst-H-S-ineq} Up to a subsequence of $(x_\alpha)_{\alpha>0}$, we have  $d_g(x_0, x_\alpha) = o(\mu_\alpha)$ when $\alpha\to +\infty$. Moreover, $(\eta_\alpha \hat{u}_\alpha)$ goes weakly to $\hat{u}$ in $D_1^2(\R^n)$ and strongly in $C^{0,\beta}_{loc}(\R^n)$ for $\beta\in (0,\inf\{1,2-s\})$ where
$\hat{u}(X) = \left(\frac{k^{2-s}}{k^{2-s} + |X|^{2-s}}\right)^{\frac{n-2}{2-s}}\hbox{ for all }X\in\R^n$
with $k^{2-s} = (n-2)(n-s)K(n,s)$. In addition,
$$\lim_{R\to +\infty}\lim_{\alpha\to +\infty}\int_{M\setminus\B_{R\mu_\alpha}(x_0)}\frac{u_\alpha^{2^*(s)}}{d_g(x,x_0)^s}dv_g = 0. $$
\end{corollary}

\begin{proof} At first, we apply Theorem \ref{Th2-Sharp-Cst-H-S-ineq} with $z_\alpha = x_\alpha$. In this case, we get that
$\eta_\alpha\hat{u}_\alpha \to \hat{u}$ in $C_c^0(\R^n)$ as $\alpha\to +\infty$. This implies that  $\lim_{\alpha \to +\infty}\hat{u}_\alpha(0) = \hat{u}(0)$, but $\hat{u}_\alpha(0) = 1$ then  $\hat{u}(0) = 1$. 

\medskip\noindent Since $\hat{u}(0) = 1$ and $\|\hat{u}\|_\infty = 1 $, Then $0$ is a maximum of $\hat{u}$. 
On the other hand, we can see from the explicit form of $\hat{u}$ in Theorem \ref{Th2-Sharp-Cst-H-S-ineq} that for all $X \in \R^n$, $\hat{u}(X) \leq \hat{u}(X_0)$. Therefore $X_0 = 0$. 
Hence, we obtain, up to a subsequence of $(z_\alpha)_{\alpha>0}$, that 
\begin{equation}\label{dist:xalpha}
d_g(x_\alpha, x_0) = \mu_\alpha d_{\hat{g}_\alpha}(X_{0,\alpha}, 0) = \mu_\alpha|X_{0,\alpha}| = o(\mu_\alpha).
\end{equation} 

\medskip\noindent We now apply Theorem \ref{Th2-Sharp-Cst-H-S-ineq} with $z_\alpha = x_0$: this is possible due to \eqref{dist:xalpha}. With the change of variable $X = \mu_\alpha^{-1}\exp^{-1}_{x_0}(x)$, we write that   
$$\int_{\B_{R\mu_\alpha}(x_0)}\frac{|u_\alpha|^{2^*(s)}}{d_g(x,x_0)^s}dv_g  = \int_{\B_R(0)}\frac{|\hat{u}_\alpha|^{2^*(s)}}{|X|^s}dv_{\hat{g}_{0,\alpha}}$$
with $\hat{g}_{0,\alpha}(X) = \exp_{x_0}^*g(\mu_\alpha X)$, we get by applying the dominated convergence Theorem twice and thanks to Theorem \ref{Th2-Sharp-Cst-H-S-ineq} that
\begin{eqnarray}
\lim_{R\to +\infty}\lim_{\alpha\to +\infty}\int_{\B_{R\mu_\alpha}(x_0)}\frac{u_\alpha^{2^*(s)}}{d_g(x,x_0)^s}dv_g  &=& \lim_{R\to +\infty}\lim_{\alpha\to +\infty}\int_{\B_R(0)}\frac{|\hat{u}_\alpha|^{2^*(s)}}{|X|^s}dv_{\hat{g}_{0,\alpha}}\nonumber
\\ &=& \int_{\R^n}\frac{\hat{u}^{2^*(s)}}{|X|^s}dX=1.\label{eq:int:1}
\end{eqnarray}
Corollary \ref{Cor-Th2-Sharp-Cst-H-S-ineq} follows from this latest relation and $\|u_\alpha\|_{\crit,s} = 1 $.\end{proof}

\section{Proof of Theorem \ref{Main-Th-Sharp-Cst-H-S-ineq}}\label{sec:proofth}
In order to prove Theorem \ref{Main-Th-Sharp-Cst-H-S-ineq},  we proceed by contradiction and assume that for all $\alpha > 0$, there exists $\tilde{u}_\alpha \in H_1^2(M)$ such that
\begin{equation}\label{eqt-1-Main-Th-Sharp-Cst-H-S-ineq}
\left(\int_M\frac{|\tilde{u}_\alpha|^{2^*(s)}}{d_g(x,x_0)^s}dv_g\right)^{\frac{2}{2^*(s)}} > K(n,s)\left(\int_M|\nabla \tilde{u}_\alpha|^2dv_g + \alpha\int_M\tilde{u}_\alpha^2dv_g\right). 
\end{equation}
We proceed in several steps : 

\medskip\noindent{\bf Step 2.1:} We claim that for all $\alpha > 0$ there exists $u_\alpha \in C^{0,\beta}(M) \cap
C^{2,\theta}(M\setminus\{x_0\}),  \beta \in (0,\min(1,2-s)), \theta \in (0,1)$  such that $u_\alpha > 0$ and  verifies 
\begin{equation}\label{eqt-2-Main-Th-Sharp-Cst-H-S-ineq}
\Delta_gu_\alpha + \alpha u_\alpha = \lambda_\alpha\frac{u_\alpha^{2^*(s)-1}}{d_g(x,x_0)^s}
\end{equation}
with $\lambda_\alpha \in (0, K(n,s)^{-1}),  \lambda_\alpha = I_\alpha(u_\alpha)$ and $\int_M\frac{u_\alpha^{2^*(s)}}{d_g(x,x_0)^s}dv_g = 1$. 

\begin{proof} Given $\alpha > 0$. By \eqref{eqt-1-Main-Th-Sharp-Cst-H-S-ineq}, there exists $\tilde{u}_\alpha \in H_1^2(M)$ that verifies $I_\alpha(\tilde{u}_\alpha) <  K(n,s)^{-1}$. 
This implies that $\lambda_\alpha:=\inf_{v\in H_1^2(M)\setminus\{0\}}I_\alpha(v) <  K(n,s)^{-1}$. Hence, thanks to Jaber \cite{jaber:test:fcts} (Theorem 4, see also  Thiam \cite{Elhadji-A.T.}), we get the Claim of Step 2.1. \end{proof}

\medskip\noindent{\bf Step 2.2:} Following Druet arguments in \cite{Druet-1} (see also Hebey \cite{Hebey-1}),  we claim that there exists $C_{22} > 0$ such that for all $x \in M$ et $\alpha > 0$, we have :
\begin{equation}\label{Step-3-Main-Th-Sharp-Cst-H-S-ineq}
d_g(x_0,x)^{\frac{n}{2}-1}u_\alpha(x) \leq C_{22}.
\end{equation}

\begin{proof} We proceed by contradiction and assume that there exists a sequence $(y_\alpha)_{\alpha>0} \in M$ such that 
\begin{equation}\label{eqt-14.1-Main-Th-Sharp-Cst-H-S-ineq}
\sup_{x\in M}d_g(x_0, x)^{\frac{n}{2}-1}u_\alpha(x) = d_g(x_0,y_\alpha)^{\frac{n}{2}-1}u_\alpha(y_\alpha)
\end{equation}
and
\begin{equation}\label{eqt-14.0-Main-Th-Sharp-Cst-H-S-ineq}
\lim_{\alpha\rightarrow+\infty}d_g(x_0,y_\alpha)^{\frac{n}{2}-1}u_\alpha(y_\alpha) = +\infty.
\end{equation}
Since $M$ is compact, we then obtain that  $\lim_{\alpha\rightarrow+\infty}u_\alpha(y_\alpha) = +\infty$. Thanks to Proposition \eqref{prop-1-Th-2-Sharp-Cst-H-S-ineq} , 
we get that, up to a subsequence,  $y_\alpha \to x_0$ as $\alpha \to \infty$. 
Now, for all $\alpha>0$, we let $\hat{r}_\alpha = u_\alpha(y_\alpha)^{\frac{-2}{n-2}}$. 

\medskip\noindent We claim that for a given $\alpha > 0$ and $R > 0$, 
\begin{equation}\label{eqt-14-Main-Th-Sharp-Cst-H-S-ineq}
\int_{\B_{\hat{r}_\alpha}(y_\alpha)}\frac{u_\alpha^{2^*(s)}}{d_g(x,x_0)^s}dv_g  = \varepsilon_R + o(1),\hbox{ where }\lim_{R\rightarrow+\infty}\varepsilon_R = 0.
 \end{equation}
\noindent Indeed, we fix $\rho > 0$. Since $y_\alpha \to x_0$ et $\hat{r}_\alpha \to 0$ as $\alpha \rightarrow +\infty$
then we write, up to a subsequence of $(y_\alpha)_{\alpha>0}$, that : 
\begin{equation}\label{eqt-15-Main-Th-Sharp-Cst-H-S-ineq}
\int_{\B_{\hat{r}_\alpha}(y_\alpha)}\frac{u_\alpha^{2^*(s)}}{d_g(x,x_0)^s}dv_g = 
\int_{\B_{\hat{r}_\alpha}(y_\alpha)\cap\B_{\rho}(x_0)}\frac{u_\alpha^{2^*(s)}}{d_g(x,x_0)^s}dv_g.
\end{equation}
Given $R >0$. Thanks to Corollary \ref{Cor-Th2-Sharp-Cst-H-S-ineq}, we have that
$$\int_{\B_{\rho}(x_0)\setminus\B_{R\mu_\alpha}(x_0)}\frac{u_\alpha^{2^*(s)}}{d_g(x,x_0)^s}dv_g = \varepsilon_R + o(1), $$
where the function $\varepsilon_R : \R \to \R$ verifies  $\lim_{R\rightarrow+\infty}\varepsilon_R = 0$. Therefore, 
\begin{eqnarray}\label{eqt-15.1-Main-Th-Sharp-Cst-H-S-ineq}
\int_{\B_{\hat{r}_\alpha}(y_\alpha)}\frac{u_\alpha^{2^*(s)}}{d_g(x,x_0)^s}dv_g &=& 
\int_{\B_{\hat{r}_\alpha}(y_\alpha)\cap\B_{\rho}(x_0)}\frac{u_\alpha^{2^*(s)}}{d_g(x,x_0)^s}dv_g \nonumber
\\ &=& \int_{\B_{\hat{r}_\alpha}(y_\alpha)\cap\left(\B_{\rho}(x_0)\setminus\B_{R\mu_\alpha}(x_0)\right)}\frac{u_\alpha^{2^*(s)}}{d_g(x,x_0)^s}dv_g\nonumber
\\ && + \int_{\B_{\hat{r}_\alpha}(y_\alpha)\cap\B_{R\mu_\alpha}(x_0)}\frac{u_\alpha^{2^*(s)}}{d_g(x,x_0)^s}dv_g\nonumber
\\ &\leq& \varepsilon_R + o(1) + \int_{\B_{\hat{r}_\alpha}(y_\alpha)\cap\B_{R\mu_\alpha}(x_0)}\frac{u_\alpha^{2^*(s)}}{d_g(x,x_0)^s}dv_g,
\end{eqnarray}
where the function $\varepsilon_R : \R \to \R$ verifies  $\lim_{R\rightarrow+\infty}\varepsilon_R = 0$
We distinguish two cases :

\medskip\noindent{\it Case 2.2.1 :} $\B_{\hat{r}_\alpha}(y_\alpha)\cap\B_{R\mu_\alpha}(x_0) = \phi$. In this case, we obtain  
immediately \eqref{eqt-14-Main-Th-Sharp-Cst-H-S-ineq} from \eqref{eqt-15.1-Main-Th-Sharp-Cst-H-S-ineq}.

\medskip\noindent{\it Case 2.2.2} :  $\B_{\hat{r}_\alpha}(y_\alpha)\cap\B_{R\mu_\alpha}(x_0) \neq \phi$. In this case, we obtain that 
\begin{equation}\label{eqt-16-Main-Th-Sharp-Cst-H-S-ineq}
d_g(x_0, y_\alpha) \leq \hat{r}_\alpha + R\mu_\alpha.
\end{equation}
By \eqref{eqt-14.0-Main-Th-Sharp-Cst-H-S-ineq}, we get that 
\begin{equation}\label{eqt-17-Main-Th-Sharp-Cst-H-S-ineq}
\lim_{\alpha\rightarrow+\infty}\frac{\hat{r}_\alpha}{d_g(x_0, y_\alpha)} = 0.
\end{equation}
Together, relations \eqref{eqt-16-Main-Th-Sharp-Cst-H-S-ineq} and \eqref{eqt-17-Main-Th-Sharp-Cst-H-S-ineq} give that 
\begin{equation}\label{eqt-18-Main-Th-Sharp-Cst-H-S-ineq}
\frac{\hat{r}_\alpha}{\mu_\alpha} =  o(1) \ {\rm and} \ d_g(x_0, y_\alpha) = O(\mu_\alpha).
\end{equation}
Independently, we consider an exponential chart $(\Omega_0, \exp_{x_0}^{-1})$ centered 
at $x_0$ such that $\exp_{x_0}^{-1}(\Omega_0) = \B_{R_0}(0), R_0 \in (0, i_g(M))$.  Under the same assumptions of Theorem \ref{Th2-Sharp-Cst-H-S-ineq}, 
we assume that $z_\alpha = x_0$ and we let $\tilde{Y}_\alpha = \mu_\alpha^{-1}\exp_{x_0}^{-1}(y_\alpha)$
and $\hat{g}_{0,\alpha} : X \in \B_{R_0}(0) \mapsto   \exp_{x_0}^*g(\mu_\alpha X)$.

\medskip\noindent By compactness arguments, there exists a constant $C_{23} > 1$ such that for all $X, Y \in \R^n$, $\mu_\alpha|X|, \mu_\alpha|Y| < R_0$ :
$$C_{23}^{-1}|X-Y| \leq d_{\hat{g}_{0,\alpha}}(X,Y) \leq  C_{23}|X-Y|.$$ 
Then we have :
\begin{equation}\label{eqt-19-Main-Th-Sharp-Cst-H-S-ineq}
|\tilde{Y}_\alpha| = O(1)    \ {\rm and} \ \mu_\alpha^{-1}\exp_{x_0}^{-1}(\B_{\hat{r}_\alpha}(y_\alpha)) \subseteq \B_{C_{23}\frac{\hat{r}_\alpha}{\mu_\alpha}}(\tilde{Y}_\alpha).
\end{equation}
Using \eqref{eqt-19-Main-Th-Sharp-Cst-H-S-ineq} and the change of variable $ X = \mu_\alpha^{-1}\exp_{x_0}^{-1}(x)$, we obtain : 
$$\int_{\B_{\hat{r}_\alpha}(y_\alpha)\cap\B_{R\mu_\alpha}(x_0)}\frac{u_\alpha^{2^*(s)}}{d_g(x,x_0)^s}dv_g \leq  \int_{\B_{C_{23}\frac{\hat{r}_\alpha}{\mu_\alpha}}(\tilde{Y}_\alpha)}
\frac{\hat{u}_\alpha^{2^*(s)}}{d_{\hat{g}_{0,\alpha}}(X,0)^s}dv_{\hat{g}_0}. $$
By dominated convergence Theorem, it follows that $$\int_{\B_{\hat{r}_\alpha}(y_\alpha)\cap\B_{R\mu_\alpha}(x_0)}\frac{u_\alpha^{\crit}}{d_g(x,x_0)^s}dv_g = o(1). $$
Therefore,  from the last relation and \eqref{eqt-15.1-Main-Th-Sharp-Cst-H-S-ineq}, we get \eqref{eqt-14-Main-Th-Sharp-Cst-H-S-ineq}. This ends the proof in the Case 2.2.2.

\medskip\noindent Now, we consider a family $(\Omega_\alpha, \exp_{y_\alpha}^{-1})_{\alpha>0}$ of exponential charts centered at $y_\alpha$ and 
we define on $\B_{R_0\hat{r}_\alpha^{-1}}(0) \subset \R^n$, $R_0 \in (0, i_g(M))$ the function $\bar{u}_\alpha(X) = \hat{r}_\alpha^{\frac{n}{2}-1}u_\alpha(\exp_{y_\alpha}(\hat{r}_\alpha X))$ and the metric 
$\bar{g}_\alpha(X) = \exp_{y_\alpha}^*g(\hat{r}_\alpha X)$. Using the same arguments of Step 1.2, we prove that there exists $\bar{u} \in D_1^2(\R^n)$ such that $\bar{u}_\alpha \to \bar{u}$ weakly in $D_1^2(\R^n)$ as $\alpha \to +\infty$.
To prove that $\bar{u}$ is non vanishing, we need the following Lemma : 

\begin{lemma}\label{lemma-Step2-Main-Th-Sharp-Cst-H-S-ineq}
The sequence $(\bar{u}_\alpha)_{\alpha>0}$ is $C^0$-bounded on any compact in $\R^n$. 
\end{lemma}
Indeed, by \eqref{eqt-14.1-Main-Th-Sharp-Cst-H-S-ineq} we have that  
\begin{equation}\label{eqt-1-lemma-Step2-Main-Th-Sharp-Cst-H-S-ineq}
\bar{u}_\alpha(X) \leq \left(\frac{d_g(x_0, y_\alpha)}{d_g(x_0, \exp_{y_\alpha}(\hat{r}_\alpha X))}\right)^{\frac{n}{2}-1}
\end{equation}
for all $X \in \B_{R_0\hat{r}_\alpha^{-1}}(0)$. Given $R > 0$, we get for all $X \in \B_R(0)$ that 
$$d_g(x_0, \exp_{y_\alpha}(\hat{r}_\alpha X)) \geq d_g(x_0, y_\alpha) -  R\hat{r}_\alpha.$$
By \eqref{eqt-1-lemma-Step2-Main-Th-Sharp-Cst-H-S-ineq} and the last triangular inequality, we get  for all $X \in \B_R(0)$ that 
\begin{equation}\label{eqt-2-lemma-Step2-Main-Th-Sharp-Cst-H-S-ineq}
\bar{u}_\alpha(X) \leq \left(1 - R\frac{\hat{r}_\alpha}{d_g(x_0, y_\alpha)}\right)^{\frac{-2}{n-2}}.
\end{equation}
By \eqref{eqt-14.0-Main-Th-Sharp-Cst-H-S-ineq}, we have that $d_g(x_0, y_\alpha)^{-1}\hat{r}_\alpha = o(1)$. 
Combining this last relation with \eqref{eqt-2-lemma-Step2-Main-Th-Sharp-Cst-H-S-ineq}, we get for all $X \in \B_R(0)$, that $\bar{u}_\alpha(X) \leq 1 + o(1)$
 in $C^0(\B_R(0))$. This ends the proof of the Lemma.  

\medskip\noindent Since $\bar{u}_\alpha(0) = 1$ for all $\alpha > 0$ then using the same arguments of Step 1.1 and Lemma \ref{lemma-Step2-Main-Th-Sharp-Cst-H-S-ineq}, we obtain by 
Theorem 4.1 in Han-Lin \cite{Han-Lin}  (see also Lemma \ref{lemma-1-Appendix-Sharp-Cst-H-S-ineq} in the Appendix) that there exists $C_{24} > 0, r > 0$ independent of $\alpha$ such that 
$\|\bar{u}_\alpha\|_{L^2(\B_r(0))} \geq C_{24}$. Letting $\alpha \to +\infty$ in the last relation, we deduce that $\bar{u} \not\equiv 0$. Similarly, we prove as in Step 1.7 that $\hat{u}_\alpha \to \hat{u}$ in $C^0_{loc}(\R^n)$.

\medskip\noindent Coming back to \eqref{eqt-14-Main-Th-Sharp-Cst-H-S-ineq}, we write that for any $\alpha > 0$ that 
\begin{equation}\label{eqt-19.1-Main-Th-Sharp-Cst-H-S-ineq}
 \int_{\B_1(0)}\frac{\bar{u}_\alpha^{2^*(s)}}{d_{\bar{g}_\alpha}(X, X_{0,\alpha})^s}dv_{\bar{g}_\alpha} = \int_{\B_{\hat{r}_\alpha}(y_\alpha)}\frac{u_\alpha^{2^*(s)}}{d_g(x,x_0)^s}dv_g = o(1) + \varepsilon_R,  
\end{equation}  
where the function $\varepsilon_R : \R \to \R$ verifies  $\lim_{R\rightarrow+\infty}\varepsilon_R = 0$. Letting $\alpha \to +\infty$ then $R \to +\infty$  in the last relation, we then get : 
$\int_{\B_1(0)}|X|^{-s}\bar{u}^{\crit}dX = 0$. Contradiction since $\bar{u} \in C^0(\B_1(0))$ and $\bar{u}(0) = \lim_{\alpha\to 0}\bar{u}_\alpha(0) = 1$. This ends Step 2.2.\end{proof}

\medskip\noindent{\bf Step 2.3:} Here goes the final argument (we adapt the one in Druet \cite{Druet-1} and in Hebey \cite{Hebey-1} to our case). We fix $\rho \in (0, i_g(M))$ sufficiently small. We consider a smooth cut-off function $\eta$ on $M$ such that 
$0 \leq \eta \leq 1$, $\eta \equiv 1$ on $\B_\rho(x_0)$ and $\eta \equiv 0$ on  $M\setminus\B_{2\rho}(x_0)$. We define the function $\eta_0$ on $\B_{2\rho}(x_0)$ by 
$\eta_0 = \eta\circ\exp_{x_0}^{-1}$. We let $dx = (\exp_{x_0}^{-1})^*dX$ and $\tilde{\delta}_0 = (\exp_{x_0}^{-1})^*\delta$.
We consider two constants, $C_{25}, C_{26} > 0$, independents of $\alpha$ such that 
$|\nabla\eta_0|_g \leq C_{25}$ et $|\Delta_g\eta_0|_g \leq C_{26}$.
The sharp Euclidean Hardy-Sobolev inequality gives for all $\alpha > 0$ that 
$$\left(\int_{\R^n}\frac{|\eta(u_\alpha\circ\exp_{x_0})|^{2^*(s)}}{|X|^s}dX\right)^{\frac{2}{2^*(s)}} 
\leq K(n,s)\int_{\R^n}|\nabla(\eta(u_\alpha\circ\exp_{x_0}))|_\delta^2dX. $$
This implies that for all $\alpha > 0$ we have : 
\begin{equation}\label{eqt-5.1-Main-Th-Sharp-Cst-H-S-ineq}
\left(\int_M\frac{|\eta_0u_\alpha|^{2^*(s)}}{d_{\tilde{\delta}_0}(x,x_0)^s}dx\right)^{\frac{2}{2^*(s)}} 
\leq K(n,s)\int_M|\nabla(\eta_0u_\alpha)|_{\tilde{\delta}_0}^2dx.
\end{equation}
In order to get a contradiction, we estimate the RHS (respectively the LHS) of the Equation \eqref{eqt-5.1-Main-Th-Sharp-Cst-H-S-ineq}, by comparing 
the $L^2$- norm of  $|\nabla(\eta_0u_\alpha)|_{\tilde{\delta}_0}$ (resp. the $L^{\crit}$- norm of  $\eta_0u_\alpha$) with respect to $\tilde{\delta}_0$ with the 
$L^2$- norm of  $|\nabla u_\alpha|_g$ (resp. the $L^{\crit}$- norm of  $u_\alpha$) with respect to $g$. We let $r_0(x) = d_g(x, x_0)$ be the 
geodesic distance to $x_0$. Cartan's expansion of the metric $g$ (see \cite{Lee-parker}) in the exponential chart $(\B_{2\rho}(x_0), exp_{x_0}^{-1})$ yields
\begin{eqnarray}\label{eqt-5.2-Main-Th-Sharp-Cst-H-S-ineq}
\int_M|\nabla(\eta_0u_\alpha)|_{\tilde{\delta}_0}^2dx &=& \int_M(1 + C_{27}r_0^2(x))|\nabla(\eta_0u_\alpha)|_g^2(1 + C_{28}r_0^2(x))dv_g\nonumber
\\ &\leq& \int_M|\nabla(\eta_0u_\alpha)|_g^2dv_g + C_{29}\int_Mr_0^2(x)|\nabla(\eta_0u_\alpha)|_g^2dv_g \nonumber
\\ &\leq&   \int_M|\nabla(\eta_0u_\alpha)|_g^2dv_g + \int_Mr_0^2\eta_0^2|\nabla u_\alpha|_g^2dv_g \nonumber
\\ &&  + C_{30}\int_{M\setminus\B_{\rho}(x_0)}u_\alpha^2dv_g,
\end{eqnarray}
where $C_i > 0, i = 27, \ldots, 30$ are independent of $\alpha$.
Independently, we get by integrating by parts that 
\begin{eqnarray}\label{eqt-5.3-Main-Th-Sharp-Cst-H-S-ineq}
\int_M|\nabla(\eta_0u_\alpha)|_g^2dv_g &=& \int_M\eta_0^2|\nabla u_\alpha|_g^2dv_g + \int_M\eta_0\Delta_g\eta u_\alpha^2dv_g
\\ &\leq& \int_M|\nabla u_\alpha|_g^2dv_g + C_{26} \int_Mu_\alpha^2dv_g
\end{eqnarray}
We let now $f_0 := \eta_0^2r_0^2$ which is a smooth function. So that
\begin{equation}\label{eqt-5.6-Main-Th-Sharp-Cst-H-S-ineq}
\int_M\eta_0^2r_0^2|\nabla u_\alpha|_g^2dv_g = \int_M(\nabla(f_0u_\alpha) -  u_\alpha\nabla f_0, \nabla u_\alpha)_gdv_g.
\end{equation}
Multiplying equation \eqref{eqt-2-Main-Th-Sharp-Cst-H-S-ineq} by $f_0u_\alpha$ then integrating by parts over $M$, we get :
\begin{eqnarray}\label{eqt-5.7-Main-Th-Sharp-Cst-H-S-ineq}
\int_M(\nabla(f_0 u_\alpha), \nabla u_\alpha)_gdv_g &=& \int_M(\Delta_gu_\alpha)f_0 u_\alpha dv_g \nonumber
\\ &\leq& \lambda_\alpha\int_M\frac{f_0 u_\alpha^{2^*(s)}}{d_g(x,x_0)^s}dv_g.
\end{eqnarray}
By Step 2.2, there exists a constant $C_{31} > 0$ independent of  $\alpha$ such that we have for all $x \in M$ : 
\begin{equation}\label{eqt-5.8-Main-Th-Sharp-Cst-H-S-ineq} 
u_\alpha^{\crit}(x)d_g(x,x_0)^{2-s}  \leq C_{31}u_\alpha^2(x).
\end{equation}
Since $\lambda_\alpha \in (0, K(n,s)^{-1})$ then by \eqref{eqt-5.7-Main-Th-Sharp-Cst-H-S-ineq} et \eqref{eqt-5.8-Main-Th-Sharp-Cst-H-S-ineq}, we get : 
\begin{equation}\label{eqt-5.9-Main-Th-Sharp-Cst-H-S-ineq}
\int_M(\nabla(f_0 u_\alpha), \nabla u_\alpha)_gdv_g \leq C_{32}\int_Mu_\alpha^2(x)dv_g,
\end{equation}
where $C_{32} > 0$ is a constant  independent of $\alpha$.
Integrating by parts gives 
\begin{equation}\label{eqt-5.10-Main-Th-Sharp-Cst-H-S-ineq}
\int_M(\nabla f_0, \nabla u_\alpha)_gu_\alpha dv_g  = \int_M\frac{1}{2}u_\alpha\Delta_g(f_0)dv_g \leq C_{33}\int_Mu_\alpha^2dv_g,\end{equation}
where $C_{33} > 0$ is a constant  independent of $\alpha$.
Plugging  \eqref{eqt-5.10-Main-Th-Sharp-Cst-H-S-ineq} and \eqref{eqt-5.9-Main-Th-Sharp-Cst-H-S-ineq} into \eqref{eqt-5.6-Main-Th-Sharp-Cst-H-S-ineq},
we get that 
\begin{equation}\label{eqt-5.11-Main-Th-Sharp-Cst-H-S-ineq}
\int_M\eta_0^2r_0^2|\nabla u_\alpha|_g^2dv_g \leq C_{34}\int_Mu_\alpha^2dv_g,
\end{equation}
where $C_{34} > 0$ is a constant  independent of $\alpha$. Therefore \eqref{eqt-5.2-Main-Th-Sharp-Cst-H-S-ineq} yields
\begin{equation}\label{eqt-5.12-Main-Th-Sharp-Cst-H-S-ineq}
\int_M|\nabla(\eta_0u_\alpha)|_{\tilde{\delta}_0}^2dx \leq \int_M|\nabla u_\alpha|_g^2dv_g + C_{35}\int_Mu_\alpha^2dv_g,
\end{equation}
where the constants $C_{35} > 0$ is independent of $\alpha$.

\medskip\noindent On the other hand, we know by Gau\ss's Lemma that $d_{\tilde{\delta}_0}(x,x_0) = d_g(x,x_0) = |\exp_{x_0}^{-1}(x)|$.
Writing that $dx = dv_g + (1 - \sqrt{det(g)})dx$ and thanks to Cartan's expansion (see Lee-Parker \cite{Lee-parker}), we obtain 
\begin{equation}\label{eqt-5.13-Main-Th-Sharp-Cst-H-S-ineq}
\int_M\frac{|\eta_0u_\alpha|^{2^*(s)}}{d_{\tilde{\delta}_0}(x,x_0)^s}dx \geq \int_M\frac{|\eta_0u_\alpha|^{2^*(s)}}{d_g(x,x_0)^s}dv_g  - \int_M\frac{|\eta_0u_\alpha|^{2^*(s)}}{d_g(x,x_0)^s}C_{36}r_0^2(x)dx, 
\end{equation}where the constants $C_{36} > 0$ is independent of $\alpha$.
With \eqref{eqt-5.8-Main-Th-Sharp-Cst-H-S-ineq}  and \eqref{Step-3-Main-Th-Sharp-Cst-H-S-ineq}, we have that 
\begin{eqnarray}\label{eqt-5.16-Main-Th-Sharp-Cst-H-S-ineq}
\int_M\frac{r_0^2(x)|\eta_0u_\alpha|^{2^*(s)}}{d_g(x,x_0)^s}dx  &\leq& C_{37}\int_M\left(u_\alpha(x)d_g(x,x_0)^{\frac{n}{2}-1}\right)^{\frac{2(2-s)}{n-2}}u_\alpha^2(x)dv_g\nonumber
\\ &\leq& C_{38}\int_Mu_\alpha^2(x)dv_g,
\end{eqnarray}
where $C_{37}, C_{38} > 0$ is a constant  independent of $\alpha$. Since we have for all $\alpha > 0$, we have that $\int_Md_g(x,x_0)^{-s}|\eta_0u_\alpha|^{2^*(s)}dv_g \leq 1$ and, up to a subsequence, 
that $\int_Mu_\alpha^2dv_g = o(1)$, it then follows with \eqref{eqt-5.13-Main-Th-Sharp-Cst-H-S-ineq} and \eqref{eqt-5.16-Main-Th-Sharp-Cst-H-S-ineq} that   
\begin{eqnarray}\label{eqt-5.18-Main-Th-Sharp-Cst-H-S-ineq}
\left(\int_M\frac{|\eta_0u_\alpha|^{2^*(s)}}{d_g(x,x_0)^s}dx\right)^{\frac{2}{2^*(s)}} &\geq& \left(\int_M\frac{|\eta_0u_\alpha|^{2^*(s)}}{d_g(x,x_0)^s}dv_g - C_{38}\int_Mu_\alpha^2(x)dv_g\right)^{\frac{2}{2^*(s)}} \nonumber
\\  &\geq&  \int_M\frac{|\eta_0u_\alpha|^{2^*(s)}}{d_g(x,x_0)^s}dv_g - C_{38}\int_Mu_\alpha^2(x)dv_g.
\end{eqnarray}
Now the definition of $\eta_0$ gives that 
\begin{equation}\label{eqt-5.14.0-Main-Th-Sharp-Cst-H-S-ineq}
\int_M\frac{|\eta_0u_\alpha|^{2^*(s)}}{d_g(x,x_0)^s}dv_g \geq  \int_M\frac{u_\alpha^{2^*(s)}}{d_g(x,x_0)^s}dv_g - \int_{M\setminus\B_\rho(x_0)}\frac{u_\alpha^{2^*(s)}}{d_g(x,x_0)^s}dv_g.
\end{equation}
Since $u_\alpha \to 0$ as $\alpha \to +\infty$ in $C^0_{loc}(M\setminus\{x_0\})$ (Proposition \ref{prop-1-Th-2-Sharp-Cst-H-S-ineq}), there exists 
a constant $C_{39} > 0$ independent of  $\alpha$ such that
\begin{eqnarray}\label{eqt-5.14-Main-Th-Sharp-Cst-H-S-ineq}
\int_{M\setminus\B_\rho(x_0)}\frac{u_\alpha^{2^*(s)}}{d_g(x,x_0)^s}dv_g &\leq& 
\sup_{x\in M\setminus\B_\rho(x_0)}\left(\frac{u_\alpha^{2^*(s)-2}}{d_g(x,x_0)^s}\right)\int_{M\setminus\B_\rho(x_0)}u_\alpha^2dv_g\nonumber
\\ &\leq&  C_{39}\int_Mu_\alpha^2dv_g.
\end{eqnarray}
Combining \eqref{eqt-2-Main-Th-Sharp-Cst-H-S-ineq}, \eqref{eqt-5.14.0-Main-Th-Sharp-Cst-H-S-ineq} and \eqref{eqt-5.14-Main-Th-Sharp-Cst-H-S-ineq} yields 
\begin{eqnarray}\label{eqt-5.15-Main-Th-Sharp-Cst-H-S-ineq}
\int_M\frac{|\eta_0u_\alpha|^{2^*(s)}}{d_g(x,x_0)^s}dv_g &\geq&  \frac{1}{\lambda_\alpha}\left(\int_M|\nabla u_\alpha|_g^2dv_g + \alpha\int_Mu_\alpha^2dv_g\right) \nonumber
\\ && \  - C_{39}\int_Mu_\alpha^2dv_g.
\end{eqnarray}
 Plugging  \eqref{eqt-5.15-Main-Th-Sharp-Cst-H-S-ineq}  into \eqref{eqt-5.18-Main-Th-Sharp-Cst-H-S-ineq}, and using that $\lambda_\alpha<K(n,s)^{-1}$, we obtain that 
\begin{equation}\label{eqt-5.17-Main-Th-Sharp-Cst-H-S-ineq}
\left(\int_M\frac{|\eta_0u_\alpha|^{2^*(s)}}{d_g(x,x_0)^s}dx\right)^{\frac{2}{2^*(s)}}  \geq   K(n,s)\int_M|\nabla u_\alpha|_g^2dv_g + (\alpha K(n,s) - C_{40})\int_Mu_\alpha^2dv_g,
\end{equation}
where $C_{40} > 0$ is a constant  independent of $\alpha$. Combining \eqref{eqt-5.12-Main-Th-Sharp-Cst-H-S-ineq}, \eqref{eqt-5.17-Main-Th-Sharp-Cst-H-S-ineq} with \eqref{eqt-5.1-Main-Th-Sharp-Cst-H-S-ineq}, we then get : 
\begin{equation}\label{eqt-5.19-Main-Th-Sharp-Cst-H-S-ineq}
(C_{41} - \alpha)\int_Mu_\alpha^2dv_g \geq 0,
\end{equation} 
where $C_{41} > 0$ is a constant  independent of $\alpha$. Contradiction since $\alpha \to +\infty$. This ends the proof of Theorem \ref{Main-Th-Sharp-Cst-H-S-ineq}.


\medskip\noindent{\bf Proof of \eqref{min:B0}.} We write $B:=B_0(M,g,s,x_0)$ for simplicity. It follows from \eqref{eqt-1.0-Main-Th-Sharp-Cst-H-S-ineq} and the definition \eqref{eqt-1-Th-2-Sharp-Cst-H-S-ineq} of $I_\alpha$ that 
\begin{equation}\label{maj:opt}
K(n,s)^{-1}\leq \inf_{u\in H_1^2(M)\setminus\{0\}}I_{K(n,s)^{-1}B}(u)
\end{equation}
We define the test-function sequence $(u_\epsilon)_{\epsilon>0}$ by
$$u_\epsilon(x) = \left(\frac{\epsilon^{1-\frac{s}{2}}}{\epsilon^{2-s} + d_g(x,x_0)^{2-s}}\right)^{\frac{n-2}{2-s}}\hbox{ for }n\geq 4 $$
for all $\epsilon > 0$ and $x \in M$. When $n=3$, the let $G_{x_0}$ be the Green's function for the coercive operator $\Delta_g+K(3,s)^{-1}B$, and we define $\beta_{x_0}:=G_{x_0}-\eta d_g(\cdot, x_0)^{-1}$ where $\eta$ is a cut-off function around $x_0$. Note that $\beta_{x_0}\in C^0(M)$, and the masse of $G_{x_0}$ is $\beta_{x_0}(x_0)$. We define for any $\epsilon>0$
$$u_\epsilon(x)=\eta(x)\left(\frac{\epsilon^{1-\frac{s}{2}}}{\epsilon^{2-s} + d_g(x,x_0)^{2-s}}\right)^{\frac{n-2}{2-s}}+\sqrt{\epsilon} \beta (x)\hbox{ for }n=3$$
for all $x\in M$.
It follows from \cite{jaber:test:fcts} that 
$$I_{K(n,s)^{-1}B}(u_\epsilon) =  K(n,s)^{-1}+\gamma_n \Omega_n(x_0)\theta_\epsilon+o(\theta_\epsilon)$$
when $\epsilon\to 0$, where $\gamma_n>0$ for all $n\geq 3$, $\theta_\epsilon:=\epsilon^2$ if $n \geq 5$, $\theta_\epsilon:=\epsilon^{2}\ln(\epsilon^{-1})$ if $n = 4$, $\theta_\epsilon:=\epsilon$ if $n = 3$, and
$$ \Omega_n(x_0):=K(n,s)^{-1}B - \frac{(n-2)(6-s)}{12(2n-2-s)}\hbox{Scal}_g(x_0)\hbox{ if }n\geq 4\, ;\, \Omega_3(x_0) :=-\beta_{x_0}(x_0).$$
It then follows from \eqref{maj:opt} that $\Omega_n(M,g,s,x_0)\geq 0$. This proves \eqref{min:B0}.


\section{Appendix} Following arguments as in Han and Lin \cite{Han-Lin} (see Theorem 4.1), we have that

\begin{lemma}\label{lemma-1-Appendix-Sharp-Cst-H-S-ineq} 
Let $\B_2(0)$ be the ball in $\R^n$ of center $0$ and radius $2$, $\tilde{g}$ be a Riemannian on $\B_2(0)$ and let $A = A(\tilde{g}) > 0$ be such that for all $\phi \in C_c^\infty(\B_2(0)) $, we have : 
$$\|\phi\|_{L^{2^*}_{\tilde{g}}(\B_1(0))} \leq A\|\nabla\phi\|_{L^2_{\tilde{g}}(\B_1(0))},$$
where $L^2_{\tilde{g}}$ is the Lebesgue space of $(\B_1(0), dv_{\tilde{g}})$. We consider $u \in H_1^2(\B_1(0), \tilde{g}), u \geq 0$ a.e. such that we have $ \Delta_{\tilde{g}}u \leq fu$, on $H_{1,0}^2(\B_1(0), \tilde{g})$
and $\int_{\B_1(0)}|f|^rdv_{\tilde{g}} \leq k$
with $r > \frac{n}{2}$ and $k > 0$ is a constant depending of $(M, g), f, r$ . Then $u \in L^\infty_{loc}(\B_1(0)) $. Moreover, for all $p > 0$, there exists a constant 
$C_{42} = C(n, p,r,\tilde{g}, k)$ such that for all $\theta\in ]0,1[ $ we have :
$$\sup_{\B_\theta(0)}u \leq C_{42}\frac{1}{(1-\theta)^{\frac{n}{p}}}\|u\|_{L^p_{\tilde{g}}(\B_1(0))}.$$
\end{lemma}
\noindent We use another version of this lemma adapted for compact Riemannian manifolds. 

\begin{lemma}\label{lemma-2-Appendix-Sharp-Cst-H-S-ineq} 
Let $(M,g)$ be a compact Riemannian Manifold. We consider $u \in H_1^2(M), u \geq  0$. We fix an open domain $\Omega$ of $M$ and assume that $u$ verifies 
$$ \left\{ \begin{array}{lll} \Delta_g u \leq f u, \ \ {\rm on} \ \Omega \ {\rm in \ the \ sense \ of \ distributions} \hbox{ ,} \\
                 \int_\Omega|f|^rdv_{\tilde{g}} \leq C_{43}, \ \ r > \frac{n}{2} \hbox{,} \end{array}\right.$$ 
 with $C_{43} = C_{43}(M,g,f,r)$ then for all $\omega \subset\subset \Omega$ and all $p > 0$, there exists $C_{44} = C_{44}(M,g,C_{43},p,r,\Omega,\omega) > 0$ (independent of $u$) such that 
 $$\|u\|_{L^\infty(\omega)} \leq C_{44}\|u\|_{L^p(\Omega)}. $$ 
\end{lemma}

\end{document}